\documentclass[]{amsart}

\usepackage{array,epsfig}
\usepackage{amsmath}
\usepackage{amsfonts}
\usepackage{amssymb}
\usepackage{amsxtra}
\usepackage{amsthm}
\usepackage{mathrsfs}
\usepackage{color}
\usepackage[dvipsnames]{xcolor}
\usepackage{tikz-cd}
\usepackage{mathabx,epsfig}
\usepackage{comment}
\usepackage{enumerate}
\usepackage{mathtools}
\usepackage{ytableau}
\usepackage{hyperref}

\hypersetup{
    colorlinks=true,
    linkcolor=blue,
    filecolor=magenta,      
    urlcolor=cyan,
    citecolor=teal
}
\theoremstyle{definition}

\newtheorem{theorem}{Theorem}[section]
\newtheorem{corollary}[theorem]{Corollary}
\newtheorem{lemma}[theorem]{Lemma}
\newtheorem{prop}[theorem]{Proposition}

\theoremstyle{definition}
\newtheorem{definition}[theorem]{Definition}

\newtheorem{remark}[theorem]{Remark}
\newtheorem{notation}[theorem]{Notation}

\newcommand{\bb}[1]{\mathbb{#1}}
\newcommand{\Z}{\bb{Z}}
\newcommand{\Q}{\bb{Q}}
\newcommand{\R}{\bb{R}}
\newcommand{\C}{\bb{C}}
\newcommand{\N}{\bb{N}}

\newcommand{\cV}{\mathcal{V}}
\newcommand{\cG}{\mathcal{G}}

\newcommand{\comp}{\circ}
\newcommand{\FI}{\mathrm{FI}}
\newcommand{\Hom}{\mathrm{Hom}}
\newcommand{\End}{\mathrm{End}}
\newcommand{\Irr}{\mathrm{Irr}}

\newcommand{\fig}{\mathrm{FI}_G}

\newcommand{\Conf}{\text{Conf}}
\newcommand{\UConf}{\text{UConf}}
\newcommand{\Ind}{\text{Ind}}

\newcommand{\sh}{\sharp}
\newcommand{\bp}{\clubsuit}

\newcommand{\nX}{\mathbf{n}_X}
\newcommand{\dX}{\mathbf{d}_X}
\newcommand{\kX}{\mathbf{k}_X}
\newcommand{\mX}{\mathbf{m}_X}

\newcommand{\nXbp}{\mathbf{n}_{X\bp}}
\newcommand{\dXbp}{\mathbf{d}_{X\bp}}

\newcommand{\mXbp}{\mathbf{m}_{X\bp}}

\usepackage[vcentermath]{youngtab}
\newcommand{\Y}[1]{{\tiny\yng(#1)}}

\title{Representation stability in the (co)homology of vertical configuration spaces}

\date{\today}

\author[Baron]{David Baron}
\email[David Baron]{jdf5@williams.edu}
\address{Williams College \\ 
Mathematics \& Statistics\\
Wachenheim Science Center \\ 
18 Hoxsey Street \\ 
Williamstown, MA 01267 USA}

\author[Pal]{Urshita Pal}
\email[Urshita Pal]{urshita@umich.edu}
\address{University of Michigan \\ 
Department of Mathematics \\
 	 530 Church St\\
 	 Ann Arbor MI, 48109 \\USA}

\author[Wang]{Chenglu Wang}
\email[Chenglu Wang]{chengluw@sas.upenn.edu}
\address{University of Pennsylvania \\ 
Department of Mathematics \\
 	 David Rittenhouse Lab. \\ 
209 South 33rd Street \\ 
Philadelphia, PA 19104-6395}

\author[Wilson]{Jennifer C. H. Wilson}
\email[Jennifer Wilson]{jchw@umich.edu}
\address{University of Michigan \\ 
Department of Mathematics \\
 	 530 Church St\\
 	 Ann Arbor MI, 48109 \\USA}

\author[Yang]{Chunye Yang}
\email[Chunye Yang]{chunye@umich.edu}
\address{University of Michigan \\ 
Department of Mathematics \\
 	 530 Church St\\
 	 Ann Arbor MI, 48109 \\USA}

\begin{document}
\maketitle

\begin{abstract}
   \noindent
   In this paper, we study sequences of topological spaces called ``vertical configuration spaces" of points in Euclidean space. We apply the theory of $\fig$-modules, and results of Bianchi--Kranhold, to show that their (co)homology groups are \emph{representation stable} with respect to natural actions of wreath products $S_k \wr S_n$. In particular, we show that in each (co)homological degree, the (co)homology groups (viewed as $S_k \wr S_n$-representations) can be expressed as induced representations of a specific form. Consequently, the characters of their rational (co)homology groups, and the patterns of irreducible $S_k \wr S_n$-representation constituents of these groups, stabilize in a strong sense. In addition, we give a new proof of rational (co)homological stability for unordered vertical configuration spaces, with an improved stable range. 
\end{abstract}
 \setcounter{tocdepth}{2}

\tableofcontents

\begin{section}{Introduction} \label{IntroSection}

\subsection{Classical (co)homological stability} 

A classical theme in algebraic topology is the study of (co)homological stability:  a sequence of spaces with maps $(X_n \xrightarrow{\varphi_n} X_{n+1})_{n \geq 0}$ satisfies \emph{homological stability} with \emph{stable range} $n \geq N_d$  if for all $n \geq N_d$, the map $$H_d(X_n) \xrightarrow{(\varphi_n)_*} H_d(X_{n+1})$$ induced on degree-$d$ homology is an isomorphism. 

One well-studied example of this phenomenon is the families of unordered configuration spaces of points in open manifolds.  Given a manifold $M$, we define the \emph{ordered configuration space} of $M$ on $n$ points to be  $$\Conf_n(M) := \{ (x_1, ..., x_n) \in M^n \; | \; \ x_i \neq x_j \ \text{if $i\neq j$} \}$$
topologized as a subspace of $M^n$. The manifold $M$ is called the \emph{background manifold}. We may alternatively view $\Conf_n(M)$ as the space of embeddings of the discrete set $\{1,\dots,n\}$ into $M$, and interpret an element of the space as a collection of $n$ distinct points in $M$, ``labelled''  by elements of the set $\{1, \dots, n\}$. The symmetric group $S_n$ acts on $\Conf_n(M)$ by permuting the labels.  The \emph{unordered configuration space} of $M$ on $n$ points is the orbit space  
$$\UConf_n(M) := \Conf_n(M) / S_n,$$ endowed with the quotient topology. An element of $\UConf_n(M)$ may be interpreted as a set of $n$ distinct unlabelled points in $M$. 

Arnol'd \cite{Arnold} proved homological stability of the unordered configuration spaces $\UConf_n(\R^2)$ with stable range $n \geq 2d$. McDuff \cite[Theorem 1.2]{McDuff} generalized this theorem to any connected open background manifold $M$ of dimension at least 2.  Segal \cite[Appendix to \S 5]{Segal} established a stable range for these families of $n \geq 2d$.

The ordered configuration spaces of these manifolds fail to exhibit homological stability, but Church--Farb \cite{ChurchFarb} formalized a sense in which a sequence of $S_n$-representations
may stabilize ``up to the $S_n$ action." They named this condition \emph{representation stability}. Church \cite{Church} and later Church--Ellenberg--Farb \cite{CEF} proved representation stability for the (co)homology of the families $(\Conf_n(M))_n$. 

\subsection{Representation stability} 

Let $(V_n)_n$ be a sequence of rational $S_n$--representations, and let $\phi_n: V_n \to V_{n+1}$ be maps that are compatible with the $S_n$-actions in an appropriate sense (made precise in Definition \ref{DefnConsistentSequence}).  Church--Farb proposed the following notion of stability for the sequence $(\phi_n: V_n \to V_{n+1})_n$. 

\begin{definition}[Church--Farb {\cite[Definition 1.1]{ChurchFarb}}]  \label{DefnRepStableSn} The sequence $(\phi_n: V_n \to V_{n+1})_n$ is \emph{representation stable} if, for sufficiently large $n$, the following conditions hold: 
\begin{enumerate}[(i)]
    \item \textbf{Injectivity}: The maps $\phi_n: V_n \rightarrow V_{n+1}$ are injective.
    \item \textbf{Surjectivity}: The span of the $S_{n+1}$-orbit of $\phi_n(V_n)$ equals all of $V_{n+1}$.
    \item \textbf{Multiplicity Stability}: Decompose $V_n$ into irreducible $S_n$-representations, $$V_n = \bigoplus_{\lambda} c_{\lambda,n} V(\lambda)_n$$ with multiplicities $0 \leq c_{\lambda,n} \leq \infty$. For each $\lambda$, the multiplicities $c_{\lambda,n}$ are independent of $n$.
\end{enumerate}
\end{definition}

Rational representations of $S_n$ are semisimple, and the irreducible representations are in canonical bijection with partitions of $n$. Definition \ref{DefnRepStableSn} depends on a certain (combinatorially natural) naming scheme for these irreducible representations that is uniform in $n$: for any partition $\lambda$ of an integer, there is an associated sequence of irreducible $S_n$-representations $V(\lambda)_n$ that is defined for all $n$ sufficiently large. This naming scheme is described in Section \ref{SectionRepStability} as a special case of Definition \ref{new partition}. With this notation, significantly,  Definition \ref{DefnRepStableSn} states that for a representation stable sequence $V_n$ of $S_n$-representations, the decomposition of each module $V_n$ into irreducible subrepresentations is (for $n$ sufficiently large) independent of $n$.

To illustrate these patterns, consider the degree-1 cohomology of $\Conf_n(\R^2)$. As $S_n$-representations, the cohomology groups decompose in the following ways, using the standard indexing of irreducible $S_n$-representations by Young diagrams of size $n$.  

\begin{alignat*}{4}
    H^1(\Conf_1(\R^2); \Q) &   \cong 0 \\ 
        H^1(\Conf_2(\R^2); \Q) &   \cong V_{\Y{2}}  \\ 
H^1(\Conf_3(\R^2); \Q) &   \cong V_{\Y{3}}  &&\oplus V_{\Y{2,1}} \\ 
H^1(\Conf_4(\R^2); \Q) &   \cong V_{\Y{4}} &&\oplus V_{\Y{3,1}} &&\oplus V_{\Y{2,2}} \\ 
H^1(\Conf_5(\R^2); \Q) &   \cong V_{\Y{5}} &&\oplus V_{\Y{4,1}} &&\oplus V_{\Y{3,2}} \\ 
H^1(\Conf_6(\R^2); \Q) &   \cong V_{\Y{6}} &&\oplus V_{\Y{5,1}} &&\oplus V_{\Y{4,2}} \\ 
\vdots & 
\end{alignat*}

Once $n \geq 4$, we obtain the decomposition at level $(n+1)$ from that at level $n$ by adding a single box to the top row of each Young diagram. Church--Farb express the stable decomposition as follows, by recording the portion of each Young diagram below its top row. This description holds for all $n\geq 4$. 
$$H^1(\Conf_n(\R^2); \Q)    \cong V(\varnothing)_n \oplus V(\Y{1})_n \oplus V(\Y{2})_n.$$ 

Church \cite{Church} proves that the same pattern holds in every (co)homological degree for every family of spaces $\Conf_n(M)$, with connected open orientable background manifold $M$ of dimension at least 2. 

Church--Ellenberg--Farb further formalized the theory of representation stability using an algebraic construct they called an \emph{$\FI$-module} \cite[Definition 1.1]{CEF}. They proved that if a sequence of $S_n$-representations has the structure of a ``finitely generated $\FI$-module''---a combinatorial condition that is easy to verify in many cases---then representation stability in the sense of Definition \ref{DefnRepStableSn} follows, in addition to other structural consequences. 

\subsection{Vertical configuration spaces} 

Our paper focuses on a variant of configuration space of $\R^s$, called \emph{vertical configuration space}, that arises from the theory of certain coloured operads (see Kranhold \cite{KranholdThesis}). These spaces were  introduced by Herberz \cite{Herberz} and Rosner \cite{Rosner} in their respective Bachelor's theses, and further studied by Latifi \cite{Latifi}, Kranhold \cite{KranholdThesis}, and Bianchi--Kranhold \cite{Vconf}. The following definitions are stated as in Bianchi--Kranhold \cite{Vconf}.

\begin{definition}[{Bianchi--Kranhold \cite[Section 1]{Vconf}}] \label{DefnVertical} For $k \geq 1$, we say a $\emph{cluster}$ of size $k$ in $\R^s$ is a $k$-tuple of $k$ distinct points of $\R^s$.
    Let $\R^s = \R^p \times \R^q$, and consider the projection map of the first $p$ coordinates, $pr : \R^s \longrightarrow \R^p$. We say a cluster $z = (z^1, \dots , z^k)$ of $k$ points in $\R^s$ is $\emph{vertical}$ if $pr(z^1) = \dots= pr(z^k)$, i.e. the $k$ points in the cluster share their first $p$ coordinates.  
\end{definition}

\begin{notation}
    For an $r$-tuple $K=(k_1,\dots,k_r)$ of positive integers,  write $|K|= k_1 + \dots + k_r$. Let $\Conf_K(\R^s)$ denote the space (canonically isomorphic to $\Conf_{|K|}(\R^s)$) of $r$-tuples $(z_1,\dots,z_r) \in (\R^s)^{k_1} \times \dots \times (\R^s)^{k_r}$ of $r$ ordered, pairwise disjoint clusters $z_i =(z_i^1,\dots,z_i^{k_i})$. 
\end{notation}

In other words, $\Conf_K(\R^s)$ is the space $\Conf_{|K|}(\R^s)$ with the additional data of a partition of the set of labels into $r$ parts of sizes $k_1,\dots,k_r$. An element in $\Conf_K(\R^s)$ is a collection of $|K|$ distinct labelled points in $\R^s$, correspondingly partitioned into $r$ clusters. 

\begin{definition}[{Bianchi--Kranhold \cite[Definition 1.1]{Vconf}}] \label{DefnVConf} Fix $s=p+q$. 
Let $r\geq 0$ and $K=(k_1,\dots,k_r)$ be an ordered tuple of positive integers.  Define the \emph{vertical configuration space}, denoted $\widetilde{\cV}_K(\R^{p,q})$, to be the following subspace of $\Conf_{K}(\R^s)$. An $r$-tuple of clusters $(z_1,\dots,z_r)$  belongs to $\widetilde{\cV}_K(\R^{p,q})$ if and only if each cluster $z_i =(z_i^1,\dots,z_i^{k_i})$ is vertical in the sense of Definition \ref{DefnVertical}. 
\end{definition}
We illustrate this definition in Figures \ref{OrderedClustersLine} and \ref{OrderedClusters}.  

\begin{figure}[h!]  
\begin{center}
    \includegraphics[scale=1.2]{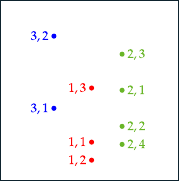}
    \caption{An element of $\tilde\cV_{({\color{red} 3}, {\color{Green} 4}, {\color{blue} 2})}(\R^{1,1})$. The 9 points are partitioned into three clusters of $3$, $4$, and $2$ points, respectively.  Points in each cluster are subject to a colinearity condition. The case $q=1$ is the origin of the name ``vertical" configuration space.  Figure adapted from Bianchi--Kranhold \cite[Figure 1]{Vconf}.} \label{OrderedClustersLine}       
        \end{center}
        \end{figure} 

\begin{figure}[h!]  
\begin{center}
    \includegraphics[scale=1.2]{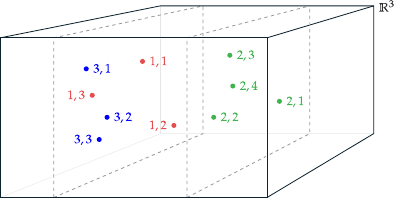}
    \caption{An element of $\tilde\cV_{({\color{red} 3}, {\color{Green} 4}, {\color{blue} 3})}(\R^{1,2})$. The 10 points are partitioned into three clusters of $3$, $4$, and $3$ points, respectively.  Points in each cluster are subject to a coplanarity condition.   Figure adapted from Bianchi--Kranhold \cite[Figure 3]{Vconf}.} \label{OrderedClusters}       
        \end{center}
        \end{figure} 

In general the full symmetric group does not act on the vertical configuration space by permuting the labels, since arbitrary permutations may not respect the verticality conditions imposed on each cluster.  Instead, if we fix $K=(k_1,\dots,k_r)$ and let $r(k)$ denote the number of occurrences of $k$ in $K$, then the product of wreath products $\prod_{k \geq 1} S_k \wr S_{r(k)}$ acts freely on $\widetilde{\cV}_K(\R^{p,q})$ by permuting clusters of the same size and labels on points within each cluster. Analogous to $\UConf$, we define the \emph{unordered} vertical configuration space as the quotient space under this action. 
\begin{definition} Fix $p,q,r\geq 0$, and let $K=(k_1,\dots,k_r)$ be an ordered tuple of positive integers.  The \emph{unordered vertical configuration space}  $\cV_K(\R^{p,q})$ is defined to be the orbit space 
 $$\cV_K(\R^{p,q}): = \widetilde{\cV}_K(\R^{p,q}) \left/ \prod_{k \geq 1} S_k \wr S_{r(k)}\right. .$$
    
\end{definition}

We may view an element of $\cV_K(\R^{p,q})$ as $|K|$ distinct unlabelled points in $\R^{p+q}$, partitioned into unordered, unlabelled clusters of sizes $k_1, k_2, \dots, k_r$, as shown in Figure \ref{UnorderedClusters}. Each cluster must be \emph{vertical} in the sense that the points in a given cluster agree in their first $p$ coordinates. 

\begin{figure}[h!]  
\begin{center}
    \includegraphics{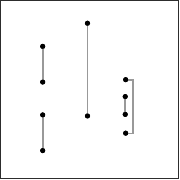}
    \caption{An element of $\cV_{(2,2,2,2,2)}(\R^{1,1})$. The 10 points in $\R^2$ are partitioned into 5 unlabelled ``vertical clusters" of 2 unlabelled points each.} \label{UnorderedClusters}       
        \end{center}
        \end{figure}

We are particularly interested in the case where $K$ is the $n$-tuple $K=(k, \dots, k)$ for some $k \geq 1$. In this case we write $\widetilde{\cV}_n^k (\R^{p,q})$ for $\widetilde{\cV}_K (\R^{p,q})$ and ${\cV}_n^k (\R^{p,q})$ for ${\cV}_K (\R^{p,q})$. The wreath product $S_k \wr S_n$ acts freely on the space $\widetilde{\cV}_n^k (\R^{p,q})$.

Latifi \cite{Latifi} proved homological stability for the families $(\cV_n^k (\R^{1,2}))_n$ for fixed $k \geq 1$. In fact, as pointed out by Bianchi--Kranhold \cite{Vconf}, Latifi's proof applies to any $p,q$ with $p+q \geq 3$. Bianchi--Kranhold \cite{Vconf} computed the cohomology groups of the ordered vertical configuration spaces, and extended the homological stability result for the unordered vertical configuration space for all $p \geq 0$, $q \geq 1$, with $(p,q) \neq (0,1)$. We review the results of Bianchi--Kranhold and Latifi  in Sections  \ref{BK-Calculation} and \ref{TransMapSection}. 

\subsection{Representation stability for vertical configuration spaces, and its consequences} 

Given the homological stability result for the unordered vertical configuration spaces, it is natural to ask: does the ordered vertical configuration space satisfy representation stability, generalizing the behaviour of $\Conf_n(\R^s)$ and $\UConf_n(\R^s)$? Our project resolves this question in the affirmative for $p\geq 0$ and $q \geq 2$. 

To answer this question, we apply the machinery  of  $\fig$-modules and $\fig\sh$-modules, as developed by Sam--Snowden \cite{SamSnowden}, Gan--Li \cite{GL}, Ramos \cite{Ramos}, Kupers--Miller \cite{KupersMiller}, and Casto \cite{casto}. This machinery generalizes the work of Church--Ellenberg--Farb to sequences of representations of wreath products $G \wr S_n$ for $G$ a fixed finite group. The following is our main theorem. 

\newtheorem*{thm:mainThmI}{Theorem \ref{mainThm}}
\begin{thm:mainThmI}
Let $p \geq 0$ and $q \geq 2$. Fix $k \geq 1$ and (co)homological degree $d \geq 0$. 
\begin{enumerate}
    \item[(\ref{MainThmI})] Let $R$ be a commutative ring.  The sequences $\left( H_d(\widetilde{\cV}^k_n (\R^{p,q});R) \right)_{n}$ and $\left( H^d(\widetilde{\cV}^k_n (\R^{p,q});R) \right)_{n }$ of $S_{k} \wr S_n$-representations are representation stable in the sense that they are $\FI_{S_k}\sh$-modules finitely generated in degree at most $ \displaystyle \left\lfloor \frac{2d}{q-1} \right\rfloor$.  
\end{enumerate}
\end{thm:mainThmI}

The theory of $\FI_{S_k}\sh$-modules implies, as a consequence, that in each (co)homological degree,  the sequence of $S_k \wr S_n$-representations $H^d(\widetilde{\cV}^k_n (\R^{p,q});R)$ (respectively, $H_d(\widetilde{\cV}^k_n (\R^{p,q});R)$) has a highly constrained form: it can be expressed as certain induced representations, as described in Corollary \ref{CorInducedReps} below. Heuristically, this means, up to the action of the wreath product $S_k \wr S_n$, all the degree-$d$ (co)homology classes ``come from'' the spaces $\widetilde{\cV}^k_n (\R^{p,q})$ for $n \leq \left\lfloor \frac{2d}{q-1} \right\rfloor$. 

\newtheorem*{thm:CorInducedReps}{Corollary \ref{CorInducedReps}}
\begin{thm:CorInducedReps}
    For $i=0, \dots \left\lfloor \frac{2d}{q-1} \right\rfloor$, there exist $S_{k} \wr S_i$-representations $U_i$ such that, as $S_{k} \wr S_n$-representations, 
$$H^d(\widetilde{\cV}^k_n (\R^{p,q});R) \cong \bigoplus_{i= 0}^{\left\lfloor \frac{2d}{q-1} \right\rfloor} \Ind_{(S_{k}\wr S_i)\times (S_{k} \wr S_{n-i})}^{S_{k} \wr S_n}(U_i\boxtimes R) \qquad \text{for all $n$, }$$ 
where $R$ denotes the trivial $(S_{k} \wr S_{n-i})$-representation. We interpret the summand $\Ind_{(S_{k}\wr S_i)\times (S_{k} \wr S_{n-i})}^{S_{k} \wr S_n}(U_i\boxtimes R)$ to be zero when $n <i$. The same statement holds for $H_d(\widetilde{\cV}^k_n (\R^{p,q});R) $. 
\end{thm:CorInducedReps}

The structure of these induced representations implies that, when we take coefficients $R=\Q$, the patterns of irreducible representations stabilize, analogous to Definition \ref{DefnRepStableSn}. The classification of irreducible representations of wreath products $S_k \wr S_n$ is reviewed in Section \ref{SectionClassificationWreathIrreps} (see Corollary \ref{ClassificationWreathIrreps}). 

\newtheorem*{thm:mainThmII}{Theorem \ref{mainThm}}
\begin{thm:mainThmII}
Let $p \geq 0$ and $q \geq 2$. Fix $k \geq 1$ and (co)homological degree $d \geq 0$. 
\begin{enumerate} 
    \item[(\ref{CorMultiplicityStability})]  The decomposition of the rational $S_k \wr S_n$-representations $H^d(\widetilde{\cV}^k_n (\R^{p,q});\Q)$ into irreducible components stabilizes in $n$ in the sense of Definition \ref{RepStabDef}, with stable range $ \displaystyle n \geq \left\lfloor \frac{4d}{q-1} \right\rfloor$. The same statement holds for $H_d(\widetilde{\cV}^k_n (\R^{p,q});\Q)$.  
\end{enumerate}
\end{thm:mainThmII}

The description of the (co)homology groups as induced $S_k \wr S_n$-representations in Corollary \ref{CorInducedReps} also has implications for the sequence of characters of these representations. Specifically, for each (co)homological degree $d$, the sequence of characters is equal to a \emph{character polynomial} (defined in Section \ref{SectionCharacterPolynomials}) that is independent of $n$. 

\newtheorem*{thm:CorHomologyCharacterPolynomial}{Corollary \ref{CorHomologyCharacterPolynomial}}
\begin{thm:CorHomologyCharacterPolynomial}
Fix $d \geq 0$. The sequence of characters of the $S_{k} \wr S_n$-representations $H^d(\widetilde{\cV}^k_n (\R^{p,q});\Q)$ can be expressed as a single character polynomial in the coloured cycle-counting functions
    $$P^{H^d(\widetilde{\cV}^k_n(\R^{p,q}))} \in \mathbb{Q}[\; X_r^{C_i} \; | \; r \geq 1, \; \text{$C_i$ a conjugacy class of $S_{k} $ }] $$ of degree at most  $\left\lfloor \frac{2d}{q-1} \right\rfloor$, independent of $n$. The same statement holds for the $H_d(\widetilde{\cV}^k_n (\R^{p,q});\Q)$. 
\end{thm:CorHomologyCharacterPolynomial}

Representation stability for the $S_k \wr S_n$-representations $H^d(\widetilde{\cV}^k_n (\R^{p,q});\Q)$ and $H_d(\widetilde{\cV}^k_n (\R^{p,q});\Q)$ means in particular that $S_k \wr S_n$-coinvariants of these representations stabilize for all $n$ sufficiently large.  In Section \ref{TransMapSection}, we explain why this implies classical stability for the (co)homology of the \emph{unordered} vertical configuration spaces ${\cV}^k_n (\R^{p,q})$. We deduce a new proof of rational cohomological stability for these spaces. Although---in contrast to the work of Latifi \cite{Latifi} and Bianchi--Kranhold \cite{Vconf}---our results only hold with rational coefficients, in the rational case we improve the stable range. 

\newtheorem*{thm:CorClassicalStability}{Corollary \ref{CorClassicalStability}}
\begin{thm:CorClassicalStability}
    Fix  $p \geq 0$, $q \geq 2$, $k \geq 1$, and (co)homological degree $d$.  The (co)homology of the unordered vertical configuration spaces, $\left( H_d({\cV}^k_n (\R^{p,q});\Q) \right)_{n}$ and $\left( H^d({\cV}^k_n (\R^{p,q});\Q) \right)_{n}$, stabilize. Specifically, the maps on homology (respectively, cohomology) induced by the stabilization maps $\phi_n: {\cV}^k_n (\R^{p,q}) \to  {\cV}^k_{n+1} (\R^{p,q})$ (respectively, the forgetful maps $\phi^n: {\cV}^k_{n+1} (\R^{p,q}) \to   {\cV}^k_{n} (\R^{p,q})$) are always injective, and are isomorphisms for $n \geq \left\lfloor \frac{2d}{q-1} \right\rfloor$. 
\end{thm:CorClassicalStability}

\subsection{Overview of the paper} 

In Section \ref{SectionWreathProducts} we review the structure of wreath products and their representation theory.  In Sections \ref{SectionCategoriesFIG}, \ref{SectionRepStability}, \ref{figsharp}, and \ref{SectionCharacterPolynomials}, we present the theory of $\fig$-modules, $\fig\sh$-modules, and their connection to representation stability. In Section \ref{VConfSection}, we apply this theory to prove the main theorems of the paper. 

\subsubsection*{Acknowledgments}

Jennifer Wilson is grateful for the support of the NSF CAREER grant DMS-2142709. David Baron, Chenglu Wang, and Chunye Yang thank the NSF and the University of Michigan for the support of their REU program (``Research Experience for Undergraduates''). David Baron was funded by NSF grant DMS-2142709. Chenglu Wang and Chunye Yang were funded by the University of Michigan. Urshita Pal is grateful for a Research Assistantship from the University of Michigan to support her work on the project and mentorship role in the REU. 

The authors thank Gahl Shemy for her mentorship during the REU. They thank Nir Gadish for helpful conversations.

\end{section}
\section{Irreducible representations of wreath products} \label{SectionWreathProducts} 

In this section we review the definition and representation theory of wreath products, as it is developed in James--Kerber \cite[Chapter 4]{JamesKerber}. 

\subsection{The structure of a wreath product $G \wr H$} 

Suppose that $X$ is a set with a faithful, transitive action by a group $G$. For $n \in \Z_{\geq 1}$ let $[n]=\{1, 2, \ldots, n\}$ and $[0]=\varnothing$. Consider the product $[n]\times X$, which we may interpret as the disjoint union of $n$ (labelled) copies of $X$. The product $G^n$ acts on $[n]\times X$, with the $i$th factor $G$ of $G^n$ acting on $\{i\}\times X$ and fixing all other copies of $X$ pointwise. Any fixed subgroup $H$ of the symmetric group $S_n$ acts on $[n]\times X$ by acting on $[n]$ by permutations and fixing $X$ pointwise; it permutes the $n$ copies of $X$. These two actions typically do not commute, but generate a group of symmetries of $[n]\times X$ called the \emph{wreath product} of $G$ by $H$ and denoted $G \wr H$. In the following definition, we give an algebraic characterization of $G \wr H$. 

\begin{definition}\label{DefWreathProduct}
    Let $G$ be a group and $H$ a subgroup of the symmetric group $S_n$. We identify the $n$-fold Cartesian product $G^n$ with the set of functions from $[n]=\{1,\dots,n\}$ into $G$:

    $$G^n=\{ \alpha \ | \ \alpha: [n] \longrightarrow G\}.$$
    The group $S_n$ acts on $[n]$ by its standard action by permutations, and acts on $G^n$ by precomposition, 
  \begin{align*} \pi: G^n &\longrightarrow G^n \\ 
  \alpha & \longmapsto  \alpha \circ \pi^{-1} .
  \end{align*} 
    
   We define the $\emph{wreath product}$ of $G$ by $H$, denoted $G\wr H$, to be the group with underlying set
    $$G^n \times H=\{(\alpha, \pi) \; |\; \ \alpha \in G^n, \pi \in H\}$$ and group operation
    $$(\alpha, \pi)(\alpha',\pi')=(\alpha\cdot \alpha'_{\pi}, \pi\pi')$$
    where $(\alpha \cdot \alpha')(i)$ denotes the pointwise product $\alpha(i) \alpha'(i)$ with respect to the group operation on $G$, $\pi\pi'$ is the product as elements of $H$, and $\alpha'_{\pi}=\alpha' \comp \pi^{-1}$.
\end{definition}

An element $(\alpha, \pi)$ of $G \wr H$ has inverse $(\bar{\alpha} \circ \pi, \pi^{-1})$, where $\bar{\alpha}$ is the pointwise inverse of $\alpha$, with $\bar{\alpha}(i) = \alpha(i)^{-1}$.

\begin{remark} \label{WreathViaMatrices} An alternate way to model the group $G \wr H$ is via \emph{monomial matrices}. Realize the group $H \subseteq S_n$ as a subgroup of $n \times n$ permutation matrices. Consider the set of all matrices obtained by replacing each of the 1's in these permutation matrices by an element of $G$. Then the group $G \wr H$ consists of these matrices, with the group product defined by matrix multiplication and the group operation on $G$.  
\end{remark} 

We introduce notation for the distinguished subgroups of $G \wr H$ isomorphic to $G^n$ and $H$. With this notation, we state Lemma \ref{WreathProductStructure} on the basic structure of the wreath product. 

\begin{definition}
     Let $G$, $H$ be as in Definition \ref{DefWreathProduct}. Let $1_H$ denote the identity permutation in $H$. Let $e$ denote the identity element of $G^n$, the constant function at the identity $1_G$ of $G$. The \emph{base group} $G^*$ of $G \wr H$ is the subgroup $$G^*=\{ (\alpha, 1_{H}) \;|\; \alpha \in G^n\}$$
    isomorphic to $G^n$. The $\emph{complement}$ of $G^*$ is the subgroup $$H'=\{(e, \pi) \;|\; \pi \in H\}$$
    isomorphic to $H$. 
\end{definition}

\begin{lemma}[{James--Kerber \cite[4.1.15]{JamesKerber}}] \label{WreathProductStructure}
     Let $G$, $H$ be as in Definition \ref{DefWreathProduct}. Then,
     \begin{enumerate}[(i)]
         \item As a set $G \wr H$ is equal to the product of subsets $G^* \cdot H' = \{(\alpha, 1_H)(e, \pi) \; | \; \alpha \in G^n, \pi \in H\}$. 
         \item $G^*$ is a normal subgroup of $G \wr H$.
         \item $G^* \cap H'$ is the trivial subgroup $\{(e, 1_{H})\}$.
     \end{enumerate}     
\end{lemma}

\subsection{Conjugacy classes in $G \wr S_n$} \label{SectionWreathConjClasses}

To study characters and character polynomials of the wreath product $G \wr S_n$, we review the classification of its conjugacy classes. Throughout this section, we assume that $G$ has countably many conjugacy classes $C^1,C^2,\dots$.

\begin{notation}\label{notation disjoint cycle notation}
    Fix an element $\pi \in S_n$. If the cycle type of $\pi$ has $c=c(\pi)$ cycles of lengths $\ell_1, \dots, \ell_c$, then we can uniquely express $\pi$ in cycle notation  

    $$\Big(j_1 \; \pi(j_1) \cdots \pi^{\ell_1-1}(j_1)\Big)\Big(j_2 \; \pi(j_2) \cdots  \pi^{\ell_2-1}(j_2)\Big) \cdots \Big(j_c \; \pi(j_c) \cdots  \pi^{\ell_c-1}(j_c)\Big)$$ 
    by imposing the conditions that $j_\nu \leq \pi^i (j_{\nu})$ for all $i, \nu$, and $j_1 \leq j_2 \leq \dots \leq j_c$. 
    \end{notation}

\begin{definition}[{James--Kerber \cite[4.2.1]{JamesKerber}}] 
    Let $\pi \in S_n$ be written in cycle notation as in Notation \ref{notation disjoint cycle notation} and let $(\alpha, \pi) \in G \wr S_n$. Then, to the $\nu$th cyclic factor $ (j_{\nu} \; \pi(j_{\nu}) \; \cdots \; \pi^{\ell_{\nu}-1}(j_{\nu}))$, we associate the element of $G$
    \begin{equation*}
        g_{\nu}(\alpha,\pi)=\alpha(j_\nu)\alpha(\pi^{-1}(j_\nu))\cdots \alpha(\pi^{-\ell_\nu+1}(j_\nu))
    \end{equation*}
    which we call the \emph{$\nu$th cycle product} of $(\alpha,\pi)$.
\end{definition}
James--Kerber show that the conjugacy class of $(\alpha, \pi)$ in $G \wr S_n$ is indexed by the data of the lengths $\ell_{\nu}$ of each cycle of $\pi$, along with the conjugacy class of 
its cycle product $g_{\nu}(\alpha,\pi)$. 
\begin{definition}[{James--Kerber \cite[4.2.2]{JamesKerber}}] \label{DefnWreathConjClasses}
    For $(\alpha,\pi) \in G \wr S_n$,  let $a_{ik}(\alpha,\pi)$ be the number of cycles of length $k$  of $\pi$ for which the associated cycle product $ g_{\nu}(\alpha,\pi)$ belongs to the conjugacy class $C^i$ of $G$. Then, the \emph{type} of $(\alpha,\pi)$ is the matrix 
   $$ a(\alpha,\pi)=\big[a_{ik}(\alpha,\pi)\big].$$
\end{definition}
\begin{theorem}[{James--Kerber \cite[4.2.8]{JamesKerber}}]
   Two elements $(\alpha,\pi)$ and $(\alpha',\pi')$ of $G \wr S_n$ are conjugate if and only if $a(\alpha,\pi)=a(\alpha',\pi')$.
\end{theorem}

Concretely, the conjugacy class of an element of $G \wr S_n$ is indexed by 
\begin{itemize}
    \item a cycle type, that is, an (unordered) partition of $n$, 
    \item for each cycle, a corresponding conjugacy class of $G$. 
\end{itemize}

\begin{remark}
    Note that the numbers $a_{ik}(\alpha, \pi)$ satisfy
    $$ \sum_{i,k} ka_{ik}(\alpha, \pi) = n.$$
    Conversely, given a matrix with nonnegative entries satisfying the above equation, it is not hard to construct an element of $G \wr S_n$ whose type is given by that matrix.
\end{remark}

\begin{corollary}[{James--Kerber \cite[4.2.9]{JamesKerber}}]
    Suppose $G$ has $s \in \Z_{\geq 1}$ conjugacy classes and $p(m)$ denotes the number of partitions of $m$. Then the number of conjugacy classes of $G \wr S_n$ is equal to 
   $$ \sum_{(n)}p(n_1)\cdots  p(n_s)$$
    where the sum is taken over all $s$-tuples $(n_1,\dots,n_s)$ such that $ {\displaystyle \sum_{i=1}^s n_i=n}$.
\end{corollary}

\subsection{The classification of irreducible $G \wr H$-representations} \label{SectionClassificationWreathIrreps}

In this paper we consider representations of wreath products of a finite group $G$ with the full symmetric group $S_n$. For the remainder of the discussion, we assume that $G$ is a finite group and that $\mathbb{K}$ is a splitting field for $G$ of characteristic zero. 

We will now review the classification of irreducible representations of $G \wr H$ over $\mathbb{K}$.

\begin{notation} \label{ExternalInternalTensor} Given an $n$-tuple of representations $D_1, \ldots, D_n$ of $G$ over $\mathbb{K}$, we use the notation $D_1 \boxtimes D_2 \boxtimes \dots \boxtimes D_n$ to denote the representation of $G^n$ obtained as the \emph{external} tensor product of the representations $D_i$; this is the $\mathbb{K}$-vector space $D_1 \otimes_\mathbb{K} D_2 \otimes_\mathbb{K} \dots \otimes_\mathbb{K} D_n$ with action of $\alpha \in G^n$ defined by 
$$ \alpha \cdot (d_1 \otimes d_2 \otimes \dots \otimes d_n) = (\alpha(1)\cdot d_1) \otimes (\alpha(2)\cdot d_2) \otimes \dots \otimes (\alpha(n)\cdot d_n) \quad \text{for } d_1 \otimes \dots \otimes d_n \in D_1 \otimes_\mathbb{K} D_2 \otimes_\mathbb{K} \dots \otimes_\mathbb{K} D_n.$$ 
In contrast, the \emph{internal} tensor product $U_1 \otimes U_2$ of representations of a group $G'$ is the tensor product $U_1 \otimes_{\mathbb{K}} U_2$ equipped with the diagonal action of $G'$, 
$$ g\cdot (u_1 \otimes u_2) = (g\cdot u_1) \otimes (g \cdot u_2) \quad \text{for } g \in G' \text{ and } u_1 \otimes u_2 \in U_1 \otimes_{\mathbb{K}} U_2.
$$
\end{notation} 

 Let $\{D^1,\dots, D^r\}$ be a complete list of inequivalent irreducible representations of $G$. A standard result from the representation theory of finite groups (see for example Fulton--Harris \cite[Exercise 2.36]{FultonHarris}) states that the irreducible representations of the finite product $G^* \cong G^n$ have the form  $$D_1 \boxtimes  D_2 \boxtimes \dots \boxtimes D_n, \qquad \text{
 where $D_i \in \{D^1,\dots, D^r\}$ for all $i \in \{1,\dots,n\}$.}$$
 These $r^n$ representations are irreducible, pairwise inequivalent, and constitute a complete list of irreducible representations of $G^n$.  


 \begin{definition}[James--Kerber {\cite[Section 4.3]{JamesKerber}}]
     Let $G$ be a finite group and $H$ a subgroup of $S_n$. Fix an irreducible representation $D^*=D_1 \boxtimes  D_2 \boxtimes \dots \boxtimes D_n$ of $G^*$. The \emph{type} of $D^*$ is the $r$-tuple  $\mu = (n_1, \dots, n_r)$, where $n_i$ denotes the number of factors isomorphic to $D^i$ in $D^*$.  Let $S_{\mu} \cong S_{n_1} \times \dots \times S_{n_r}$ be the subgroup of $S_n$ that stabilizes the partition of the index set $[n]$ of $D^*$ into the $n_1$ indices $i$ corresponding to factors $D_i$ equal to $D^1$, the $n_2$ indices $i$ corresponding to factors $D_i$ equal to $D^2$, etc. Let $H_{D^*}$ denote $H \cap S_{\mu}$. Let $H'_{D^*} \cong H_{D^*}$ denote the corresponding subgroup of $G \wr H$, 
     $$H'_{D^*} = \{ (e, \pi) \; | \; \pi \in  H_{D^*}\} \subseteq H' \subseteq G \wr H.$$
     The \emph{inertia subgroup} of $D^*$ is the subgroup  $G^* \cdot H'_{D^*}$ of $G \wr H$. The inertia subgroup is isomorphic to the wreath product of $G$ by $H_{D^*}$, and we denote it by $G \wr H_{D^*}$.
     \end{definition} 
 
\begin{definition}
 The \emph{conjugate} of the $G^*$-representation $D^*$  by    $(\alpha, \pi) \in G \wr H$ is the $G^*$-representation obtained by pulling back the representation along the group homomorphism $G^* \to G^*$ given by conjugation by $(\alpha, \pi)$. Two $G^*$-representations are \emph{conjugate} if one is conjugate to the other via some   $(\alpha, \pi) \in G \wr H$. 
 \end{definition} 
James--Kerber characterize the inertia subgroup $G \wr H_{D^*}$ as the set of elements $(\alpha, \pi)$ in $G \wr H$ for which the conjugate of $D^*$ by $(\alpha, \pi)$ is isomorphic to $D^*$ as a $G^*$-representation: it is the stabilizer of $D^*$ with respect to the action of $G \wr H$ by conjugation on the set of irreducible $G^*$-representations.

With these definitions, James--Kerber \cite{JamesKerber} use Clifford Theory \cite{CliffordTheory} to classify the irreducible representations of $G \wr H$. 

    \begin{theorem}[{James--Kerber \cite[Theorem 4.3.34]{JamesKerber}}] \label{theorem irreps}
 Let $H$ a subgroup of $S_n$. Fix an irreducible $G^*$ representation $D^*=D_1 \boxtimes  D_2 \boxtimes \dots \boxtimes D_n$. The action of $G^*$ on $D_1 \boxtimes  D_2 \boxtimes \dots \boxtimes D_n$ extends to an action of the inertia subgroup $G \wr {H}_{D^*}$, where 
 $$ (\alpha, \pi) \cdot (d_1 \otimes d_2 \otimes \dots \otimes d_n) = (\alpha(1)\cdot d_{\pi^{-1}(1)}) \otimes (\alpha(2)\cdot d_{\pi^{-1}(2)}) \otimes \dots \otimes (\alpha(n)\cdot d_{\pi^{-1}(n)})$$ for $(\alpha, \pi) \in G \wr {H}_{D^*}$ and $d_1 \otimes \dots \otimes d_n \in  D_1 \boxtimes   \dots \boxtimes D_n$.  
 We denote this $G\wr {H}_{D^*}$-representation by $\widetilde{D}^*$. 

 Let $D'$ be an irreducible representation of ${H}_{D^*}$. We let $\widetilde{D}'$ denote the representation of $G\wr {H}_{D^*}$ pulled back from $D'$ along the quotient map 
 \begin{align*}
     G\wr {H}_{D^*} & \longrightarrow {H}_{D^*} \\ 
     (\alpha, \pi) & \longmapsto \pi .
 \end{align*}

 Let $\widetilde{D}^* \otimes \widetilde{D}'$ denote the internal tensor product of these $G\wr {H}_{D^*}$-representations. We write $D$ for the induced representation
$$D = \mathrm{Ind}_{G\wr {H}_{D^*}}^{G \wr H} \widetilde{D}^* \otimes \widetilde{D}'.$$ 
 Then, 
        \begin{enumerate}[(i)]
            \item $D$ is irreducible and every irreducible representation of $G \wr H$ over $\mathbb{K}$ has this form.
            \item $D$ runs through a complete system of pairwise inequivalent and irreducible representations of $G \wr H$ if $D^*$ runs through a complete system of pairwise non-conjugate, irreducible representations of $G^*,$ and, while $D^*$ remains fixed, $D'$ runs through a complete system of pairwise inequivalent irreducible representations of ${H}_{D^*}$.
        \end{enumerate}
    \end{theorem}

    We remark that when $H=S_n$, two $G^*$-representations are conjugate if and only if they have the same type $\mu$  \cite[Section 4.4]{JamesKerber}.  In this case, the group $H_{D^*} = S_{\mu}$ is a product of symmetric groups. Since irreducible representations of the symmetric group $S_n$ are in canonical bijection with partitions of $n$, Theorem \ref{theorem irreps} specializes as follows. 

    \begin{corollary} \label{IrrepsGwreathSn} Let $\chi_1, \ldots, \chi_r$ be a complete list of irreducible representations of a finite group $G$ over a splitting field $\mathbb{K}$ of $G$ of characteristic zero. Then the irreducible representations of $G \wr S_n$ are classified by partition-valued functions
    $\underline{\lambda}$ from $\{\chi_1, \ldots, \chi_r\}$ to partitions of nonnegative numbers, such that the sum $\sum_{i=1}^r |\underline{\lambda} (\chi_i)|$ is $n$. 
    \end{corollary}
 The functions  $\underline{\lambda}$ are sometimes called \emph{multipartitions of $n$ with $r$ components.} \\

 When $G = S_k$, we can further specialize the theorem. We note, in this case, that every irreducible representation of $S_k \wr S_n$ is defined over the rational numbers (see James--Kerber \cite[Theorem 4.4.8 and Corollary 4.4.9]{JamesKerber}).
    
\begin{corollary} \label{ClassificationWreathIrreps}
 Let $\mathbb{K}$ be $\mathbb{Q}$ (or any field of characteristic zero). Then the irreducible representations of $S_k \wr S_n$ over $\mathbb{K}$ are in canonical bijection with functions $\underline{\lambda}$ from the set of partitions $\rho$ of $k$ to the set of all partitions, such that $\sum_{\rho} |\underline{\lambda}(\rho)|=n.$ Explicitly, if $W_{\rho}$ denotes the $S_k$-representation associated to $\rho$, then (in the notation of Theorem \ref{theorem irreps}) the function $\underline{\lambda}$ corresponds to the irreducible $S_k \wr S_n$-representation
$$\mathrm{Ind}_{S_k \wr \left(\prod_{\rho} S_{n_{\rho}}\right)}^{S_k \wr S_n} \left( \widetilde{\bigboxtimes_{\rho} W_{\rho}^{\boxtimes n_{\rho}} } \right) \otimes  \left( \widetilde{\bigboxtimes_{\rho} W_{\underline{\lambda}(\rho)}}\right)$$
 where $n_{\rho}$ denotes the size $|\underline{\lambda}(\rho)|$ of the partition $\underline{\lambda}(\rho)$. 
    
\end{corollary}

\subsection{The branching rule for $G\wr S_n$}

Let $G$ be a finite group and $\mathbb{K}$ be a splitting field of $G$ of characteristic zero. From the construction of the irreducible representations of $G\wr S_n$ and the Littlewood--Richardson rule for the symmetric groups, we can describe the structure of certain induced representations of wreath products. See, for example, Stein \cite[Theorem 4.5]{littlewood-richardson}, or Ingram--Jing--Stitzinger \cite[Theorem 4.3]{IJS-Wreath}. 

 \begin{theorem}[E.g. {\cite[Theorem 4.5]{littlewood-richardson}}]  Let $\chi_1, \ldots, \chi_r$ be a complete list of the $r$ irreducible representations of $G$ over $\mathbb{K}$. Fix integers $0 \leq d \leq n$. Let $\underline{\delta} = (\delta_1, \dots, \delta_r)$ be a multipartition of $d$ with $r$ components, and let $\underline{\rho} = (\rho_1, \dots, \rho_r)$ be a multipartition of $(n-d)$ of $r$ components. Let $W_{\underline{\delta}}$ and $W_{\underline{\rho}}$ denote the corresponding irreducible representations of $G \wr S_d$ and $G \wr S_{n-d}$, respectively. 

 $$ \Ind_{G \wr S_d  \times G \wr S_{n-d}}^ {G \wr S_n} W_{\underline{\delta}} \boxtimes W_{\underline{\rho}} \quad \cong \quad \bigoplus_{\underline{\lambda}} \left(  \prod_{i=1}^r C_{\rho_i, \delta_i}^{\lambda_i} \right) W_{\underline{\lambda}}  \ .$$ 
 Here, the sum is taken over all multipartitions $\underline{\lambda}$ of $n$ with $r$ components satisfying $|\lambda_i| = |\rho_i| + |\delta_i|$ for $i=1, \dots, r$. The constants $C_{\rho_i, \delta_i}^{\lambda_i}$ are the Littlewood--Richardson coefficients. 
 \end{theorem}

 To obtain an analogue of Pieri's rule, we specialize to the case that $W_{\underline{\rho}}$ is the trivial $G \wr S_{n-d}$-representation. We let $\chi_1$ denote the trivial $G$-representation, so $W_{\underline{\rho}}$ corresponds to the multipartition $\underline{\rho} = ( (n-d), \varnothing,  \varnothing, \dots, \varnothing).$  We deduce the following. 
\begin{corollary} \label{PieriRuleWreath} Let $\chi_1, \ldots, \chi_r$ be a complete list of the $r$ irreducible representations of $G$ over $\mathbb{K}$, with $\chi_1$ the trivial representation. Let $\underline{\delta} = (\delta_1, \dots, \delta_r)$ be a multipartition of $d$ with $r$ components, and let $W_{\underline{\delta}}$ denote the corresponding irreducible representation of $G \wr S_d$.  Then, 
    
\begin{align*}
     \Ind_{G \wr S_d  \times G \wr S_{n-d}}^ {G \wr S_n} W_{\underline{\delta}} \boxtimes \text{trivial} \quad &\cong \quad \bigoplus_{\substack{\text{partitions $\lambda_1$ of }\\ |\delta_1|+(n-d)}}  C_{\delta_1, (n-d)}^{\lambda_1} W_{(\lambda_1, \delta_2, \delta_3, \dots, \delta_r)} \\
      &\cong \quad \bigoplus_{\substack{\text{Young diagrams $\lambda_1$ obtained} \\ \text{from $\delta_1$ by adding $(n-d)$} \\ \text{boxes in distinct columns}}} W_{(\lambda_1, \delta_2, \delta_3, \dots, \delta_r)}. 
\end{align*}
 \end{corollary}




 

\section{The category $\fig$ and $\fig$-modules} \label{SectionCategoriesFIG}

The field of \emph{representation stability} was developed by  Church, Ellenberg, Farb, and Nagpal \cite{ChurchFarb, CEF, CEFN, CE-FIhomology}, Putman, Sam, and Snowden \cite{Snowden-Segre, PutmanCentralStability,  SS-StabilityPatterns, SS-Grobner, PutmanSam-VIC}, Gan and Li \cite{GL, GL-Linear},   and many others. See also, for example, \cite{SS-IntroTCA, Farb-Survey, JennyNotes, Sam-notes, Snowden-MSRI, Sam-Survey, JRW-Survey} for surveys and notes. 

Church--Ellenberg--Farb developed the theory of $\FI$-modules as a formal categorical framework for studying sequences of representations of the symmetric groups. The foundations of this theory are developed in their paper \cite{CEF}.  In this section, we describe a generalization of this framework to sequences of representations of wreath products $G \wr S_n$, for a fixed finite group $G$. This generalization has been developed by Sam--Snowden \cite{SamSnowden}, Gan--Li \cite{GL}, Ramos \cite{Ramos}, Kupers--Miller \cite{KupersMiller}, Casto \cite{casto},  and others. The special case of $G=S_2$ was studied in earlier work \cite{ChurchFarb, Wilson-PureStringMotion, Wilson-WeylI, Wilson-WeylII}, in part because $S_2 \wr S_n$ is isomorphic to the Weyl group of type $B_n$/$C_n$. 

In the case that $G$ is the trivial group, then $\fig$ is canonically isomorphic to the category $\FI$, and the results here specialize to the results of Church--Ellenberg--Farb. 

In Sections \ref{SectionCategoriesFIG}, \ref{SectionRepStability}, \ref{figsharp}, and \ref{SectionCharacterPolynomials} we review the theory of representation stability for wreath products, the tools we will use to prove it, and some combinatorial consequences for the structure of the representations and their characters. We do not claim originality for this theory---credit is due to Sam--Snowden \cite{SamSnowden}, Gan--Li \cite{GL}, Ramos \cite{Ramos}, and  Casto \cite{casto}---and many results closely parallel those of Church--Ellenberg--Farb \cite{CEF} for $\FI$-modules.  We do review the material in detail, and establish versions of some results that are implicit in the literature but not readily citable (such as Propositions \ref{fig criterion} and \ref{characterization finitely generated fig}) that we hope may also assist future authors in applying the $\fig$-module machinery. We use a novel combinatorial characterization of the categories  $\fig$ and $\fig\sh$, which we hope the reader may find conceptually or computationally useful.

\subsection{The category $\fig$} 

\begin{definition} \label{DefnFIG}  Let $G$ be a finite group. Let $X$ denote the underlying set of $G$, viewed as a faithful, transitive $G$-set with respect to the action of $G$ by right multiplication. We define a category $\fig$ as follows. The objects are products $A \times X$ of a finite set $A$ with the set $X$. The morphisms are all injective maps of the form  $f:A \times X \to B \times X$ such that if $f(a, 1_G) = (b, g)$, then $f(a, g') = (b, gg')$ for all $a \in A$ and $g' \in G$.  
\end{definition}  

We may view the object $A \times X$ as the disjoint union of $|A|$ copies of $X$, each labelled by an element of $A$.  A morphism $A \times X \to B \times X$, then, is an injective map $\phi: A \to B$ on label sets, and on each copy $\{a\} \times X$ of $X$ a $G$-equivariant map to $\{\phi(a)\} \times X$. 

We can therefore specify the data of an $\fig$ morphism by an injective map $\phi: A \to B$ and a map $\alpha: A \to G$ that determines the values of the points $A \times \{1_G\}$. Then the pair $(\alpha, \phi)$ corresponds to the injective map
\begin{align*}
    A \times X & \longrightarrow B \times X \\ 
    (a, g) & \longmapsto (\phi(a),  \alpha(a)g). 
\end{align*}
This observation implies the following proposition, which reconciles Definition \ref{DefnFIG} with the definition of $\fig$ used by Sam--Snowden \cite[Section 1.1]{SamSnowden}, Gan--Li \cite[Definition 1.1]{GL}, Casto \cite{casto}, and Ramos \cite[Section 1]{Ramos}.
\begin{prop}
Let $G$ be a finite group. The objects of $\fig$ are indexed by finite sets $A$, and the morphisms from $A$ to $B$ correspond precisely to the set of pairs $(\alpha, \phi)$ consisting of an injective map $\phi: A \to B$ and an arbitrary map $\alpha: A \to G$, subject to the composition rule
    $$(\alpha, \phi)(\alpha',\phi')=( (\alpha 
    \circ \phi') \cdot \alpha', \phi \circ \phi')$$
    where $\cdot$ denotes the pointwise product with respect to the group operation on $G$.
\end{prop}

For the rest of the paper, we take the endomorphisms of the finite set $[n]\times X$ as the definition of the wreath product  $G \wr S_n$. We observe that it is isomorphic to the group described in Definition \ref{DefWreathProduct}.  

\begin{remark} \label{RemarkDefnFIG}
    Fix $k \in \N$. In the case that $G$ is the symmetric group $S_k$, we may replace $X$ with the faithful, transitive $S_k$-set $[k]=\{1, 2, \ldots, k\}$ equipped with the standard action of $S_k$ by permutations. We obtain the same category $\FI_{S_k}$ as in Definition \ref{DefnFIG}.
\end{remark}

\begin{remark}
    To study functors with domain $\fig$, it suffices to consider the objects of the form $[n]\times X$ for natural numbers $n \geq 0$. Thus we sometimes restrict to the full subcategory of $\fig$ (equivalent to $\fig$) with objects $[n]\times X$ indexed by the natural numbers. By abuse of notation, some authors also denote this subcategory by $\fig$. 
\end{remark}




\begin{notation}
    For brevity, for $n \in \Z_{\geq 0}$, we will sometimes write $\nX$ to denote the set $[n] \times X$. We similarly write $\dX$, $\kX$, $\mX$, \ldots, and $\mathbf{0}_X$,  $\mathbf{1}_X$, \dots. 
\end{notation}

\begin{notation} \label{NotationInclusion}
   Given $0 \leq m \leq n$,  let $\iota_{m,n}: \mX \to \nX$ denote the natural inclusion map 
$$\iota_{m,n} =  (\iota \times id_X): [m] \times X \longrightarrow [n] \times X$$
for $\iota: [m] \hookrightarrow [n]$ the inclusion.  Notably, we can factor $\iota_{m,n} = \iota_{n-1, n} \circ \dots \circ \iota_{m, m+1}$. 
\end{notation}

As in the case of $\FI$, it is convenient to observe that we can factor morphisms of $\fig$ as a composite of a natural inclusion map and an automorphism. 

\begin{prop} \label{Map Factors Lemma in fig context}
  Fix $0 \leq m \leq n$. Every morphism in $f \in \Hom_{\fig}\left(\mX, \nX\right)$ factors as a composite
$$ \mX \xrightarrow{\iota_{m,n}} \nX \overset{\omega}{\longrightarrow} \nX $$ 
for some (non-canonical) $\omega \in  \End_{\fig}(\nX) = G \wr S_n $.  

\end{prop}

In light of Proposition \ref{Map Factors Lemma in fig context}, as in the theory of $\FI$-modules, we will see that the key to understanding the combinatorics of the $\fig$ morphisms is to calculate the stabilizer of the morphism $\iota_{m,n}$ under the action of $\End_{\fig}\left(\nX \right) = G \wr S_n$ by postcomposition. 

\begin{prop} \label{PropDescribeStabilizer} Let $0 \leq m\leq n$. Consider the action of $\End_{\fig}\left(\nX \right) = G \wr S_n$ by postcomposition on the set of morphisms $\Hom_{\fig}(\mX , \nX)$. Then the stabilizer stab($\iota_{m,n}$) of $\iota_{m,n}$ is precisely the set of morphisms
$$ \nX  \longrightarrow \nX$$ that restrict to the identity on $(\{1, 2, \ldots, m\} \times X) \subseteq  ([n] \times X) = \nX$. By considering the restriction of these maps to $(\{m+1, \dots, n\} \times X) \subseteq ([n] \times X)=\nX$, we can canonically identify the stabilizer
$$ \mathrm{stab}(\iota_{m,n}) \cong \End_{\fig}(\{m+1, \dots, n\} \times X),$$
a group isomorphic to the wreath product $G \wr S_{n-m}$. 
\end{prop}

\subsection{$\fig$-modules}

\begin{definition}
    An \emph{$\fig$-module} $V$ over a commutative ring $R$ is a functor $V$ from $\fig$ to the category $R$-Mod of $R$-modules. Given a finite set $A$, for simplicity we write $V_A$ for $V(A \times X)$. We let $V_n$ denote $V(\nX)$. 
\end{definition}


\begin{definition}
    Let $R$ be a commutative ring and let $V$ and $W$ be $\fig$-modules over $R$. We define a map
    $$F:V\longrightarrow W$$
    to be a \emph{map of $\fig$-modules} if $F$ is a natural transformation of functors $V$ and $W$. More concretely, $F$ is a family of $R$-linear maps
    $$F_A:V_A\longrightarrow W_A$$
    for each finite set $A$ such that the following diagram commutes for all $\fig$ morphisms $f: A \times X \to B \times X$, 
    \[\begin{tikzcd}
	{V_A} && {W_A} \\
	\\
	{V_B} && {W_B}
	\arrow["{F_A}",from=1-1, to=1-3]
	\arrow["{f_*}", from=1-1, to=3-1]
	\arrow["{f_*}", from=1-3, to=3-3]
	\arrow["{F_B}", from=3-1, to=3-3]
\end{tikzcd}\]
where $f_*$ denotes the map induced by $f$. We say that $U$ and $V$ are isomorphic as $\fig$-modules if there is a natural isomorphism $F:U\to V$ of $\fig$-modules, that is, a natural transformation $F$ such that $F_A$ is an $R$-linear isomorphism for every finite set $A$. 
\end{definition}

\begin{definition}
    Let $V$ be an $\fig$-module. We say that $U$ is an \emph{$\fig$-submodule} of $V$ if there is a map of $\fig$-modules $F: U\to V$ such that $F_A$ is the inclusion of an $R$-submodule of $V_A$ for each finite set $A$. 
\end{definition}

\begin{definition}
    Suppose that $F:V \to U$ is a map of $\fig$-modules. Then the family of kernels ker$(F_A) \subseteq V_A$ for each finite set $A$ form an $\fig$-submodule of $V$ which we denote by ker$(F)$ and call the \emph{kernel} of $F$. The family of images im$(F_A) \subseteq U_A$ form an $\fig$-submodule of $U$, which we denote by im$(F)$ and call the \emph{image} of $F$.

Sam--Snowden \cite[Theorem 1.2.4]{SamSnowden} showed that, in a certain strong sense, the theory of $\fig$-modules reduces to the theory of $\FI$-modules. 
    
\end{definition}



\subsection{An $\fig$-module criterion} 

\begin{definition} \label{DefnConsistentSequence} Let $G$ be a finite group, and $R$ a commutative ring.
   For each integer $n \geq 0$, let $V_n$ be a $G \wr S_n$-representation over $R$, and let $\varphi_n: V_n \to V_{n+1}$ be an $R$-linear map.  View $G \wr S_n$ as the subgroup of $G \wr S_{n+1}$ of symmetries of $([n] \times X) \subseteq ([n+1] \times X)$. We call the sequence 
   $$ V_0 \xrightarrow{\varphi_0} V_1 \xrightarrow{\varphi_1}  V_2 \xrightarrow{\varphi_2} \dots \xrightarrow{\varphi_{n-1}} V_n \xrightarrow{\varphi_n}  V_{n+1} \xrightarrow{\varphi_{n+1}} \dots $$ a \emph{consistent sequence} of $G\wr S_n$-representations if the maps $\varphi_n$ are $G \wr S_n$-equivariant with respect to the action of $G \wr S_n$ on $V_n$ and the restriction of the action to   $G \wr S_n \subseteq G \wr S_{n+1}$ on $V_{n+1}$. 
\end{definition}

An $\fig$-module $V$ defines in particular a consistent sequence of $G\wr S_n$-representations $V_n$  with equivariant maps $(\iota_{n, n+1})_*: V_n \to V_{n+1}$. However, not every consistent sequence of  $G\wr S_n$-representations arises in this way. The following results show when a sequence of $G\wr S_n$-representations can be given the structure of a functor from (the full subcategory on objects $\nX$ of) the category $\fig$.   This result generalizes Church--Ellenberg--Farb \cite[Remark 3.31]{CEF}.

\begin{lemma} \label{fig criterion}\label{promoting sequences criterion} Let $G$ be a finite group, and $R$ a commutative ring. Suppose that     $$ V_0 \xrightarrow{\varphi_0} V_1 \xrightarrow{\varphi_1}  V_2 \xrightarrow{\varphi_2} \dots \xrightarrow{\varphi_{n-1}} V_n \xrightarrow{\varphi_n}  V_{n+1} \xrightarrow{\varphi_{n+1}} \dots $$ is a consistent sequence of $G \wr S_n$-representations. Let $\varphi_{m,n}:V_m \to V_n$ denote the composite  $\varphi_{n-1} \circ \dots \circ \varphi_m$.

Let $\iota_{m,n} \in \Hom_{\fig}(\mX, \nX)$ be the natural inclusion map as defined in Notation \ref{NotationInclusion}. Let stab($\iota_{m,n}$) denote
$$ \mathrm{stab}(\iota_{m,n}) = \{ \omega \in \End_{\fig}(\nX)=G \wr S_n \; | \; \omega \circ \iota_{m,n} = \iota_{m,n} \},$$
as in Proposition \ref{PropDescribeStabilizer}. 
Then there exists a functor $V: \fig \to R$-Mod satisfying $V(\nX) = V_n$ and $(\iota_{m,n})_* = \varphi_{m,n}$ if and only if the following Condition ($*$) is satisfied: 
   \begin{equation} \tag{$*$}
        \text{for all $m<n$,\quad $\omega \cdot v = v$,\quad for all $\omega \in \mathrm{stab}(\iota_{m,n})$ and $v\in \varphi_{m,n}(V_m)$.} 
    \end{equation}
\end{lemma} 

\begin{proof}
 Given an $\fig$-module $V: \fig \to R$-Mod, the sequence $(\iota_{n, n+1})_*: V_n \to V_{n+1}$ satisfies Condition ($*$)  by functoriality. Conversely, suppose that $\varphi_n: V_n \to V_{n+1}$ is a consistent sequence of $G \wr S_n$-representations that satisfies $(*)$. We will construct the functor $V$ and verify that it is functorial. 
 
 It suffices to construct $V$ on the full subcategory of $\fig$ of objects $\nX=[n] \times X$ for integers $n\geq 0$. Let $V(\nX)=V_n$, and let $(\iota_{m,n})_* = \varphi_{m,n}$.  Let $f \in \Hom_{\fig}(\mX, \nX)$.  By Proposition $\ref{Map Factors Lemma in fig context}$, we can factor $f$ (non-uniquely) as a composite $f=\sigma \circ \iota_{m,n}$, where $\sigma \in \End_{\fig}(\nX) = G \wr S_n$. Thus we define $f_*$ to be the function $\sigma \circ \varphi_{m,n}: V_m \to V_n$. 

We first check that $f_*$ is well-defined. Suppose that $f = \sigma \circ \iota_{m,n} = \tau \circ \iota_{m,n}$ in $\Hom_{\fig}(\mX, \nX)$, for $\sigma, \tau \in G\wr S_n$. Then $\tau^{-1}\sigma \in \mathrm{stab}(\iota_{m,n})$, so $\sigma \circ \varphi_{m,n} = \tau \circ \varphi_{m,n}$ by Condition $(*)$. 

We finally check that $V$ is functorial. Let $h \in \Hom_{\fig}(\kX, \mX)$ and $f \in \Hom_{\fig}(\mX, \nX)$. Factor $h = \rho \circ \iota_{k,m}$ and $f= \sigma \circ \iota_{m,n}$.  Observe that the map $\iota_{m,n}$ defines an inclusion of $G \wr S_m$ into $G \wr S_n$ as the symmetry group of $\mX = ([m]\times X) \subseteq ([n]\times X) \subseteq \nX$. For $\rho \in G \wr S_m$, write $\rho' \in G \wr S_n$ for its image; by definition $\rho'$ satisfies $\iota_{m,n}  \circ \rho = \rho'  \circ \iota_{m,n}$ as $\fig$ morphisms. Then, 
\begin{align*}
    f_* \circ h_* & = (\sigma \circ \iota_{m,n})_* \circ (\rho \circ \iota_{k,m})_* \\ 
    & = (\sigma \circ \varphi_{m,n}) \circ (\rho \circ \varphi_{k,m}) \\ 
    & = \sigma \circ \rho'\circ \varphi_{m,n} \circ \varphi_{k,m} \qquad \qquad \text{ (by the equivariance assumption on the maps $\varphi_i$)} \\ 
    & = (\sigma \circ \rho'\circ \iota_{m,n} \circ \iota_{k,m})_* \\ 
    & = (\sigma \circ \iota_{m,n} \circ \rho \circ \iota_{k,m})_* \\
    & = (f \circ h)_*  \qquad  \qquad \qquad  \qquad \qquad \text{which concludes the proof.} \qedhere
\end{align*}
\end{proof}

\subsection{Generation of $\fig$-modules and the Noetherian property} 
In this section we recall the Noetherian property of $\fig$-modules over Noetherian rings. The following definitions appear in Casto \cite[Section 2]{casto} and Sam--Snowden \cite[Section 2.1]{SamSnowden}, and are direct analogues of the corresponding definitions for $\FI$-modules \cite[Section 2.3]{CEF}. 

\begin{definition}
    An $\fig$-module $V$ is $\emph{generated}$ by a set $S \subseteq \amalg_{n \geq 0} V_n$ if $V$ is itself the smallest $\fig$-submodule of $V$ containing $S$. 
\end{definition}

\begin{definition}
    An $\fig$-module $V$ is $\emph{finitely generated}$ if it is generated by a finite set. We say $V$ is $\emph{generated in degree}$ $ \leq d$ whenever $V$ is generated by $\amalg_{0 \leq n \leq d} V_n$. If $V$ is generated in degree $d < \infty$, then we say $V$ has $\emph{finite generation degree}$. 
\end{definition}

\begin{definition}
    A finitely generated $\fig$-module $V$ is $\emph{Noetherian}$ if every $\fig$-submodule of $V$ is finitely generated.
\end{definition}

Sam--Snowden \cite{SamSnowden} proved a Noetherianity result for $\fig$-modules for \emph{virtually polycyclic groups} $G$. It holds in particular when $G$ is finite. 

\begin{theorem}[{\cite[Corollary 1.2.2 and Theorem 1.2.3]{SamSnowden}}]\label{Noetherianity for FIG}
    Let $G$ be a finite group and let $R$ be a left-Noetherian ring. Then any finitely generated $\fig$-module $V$ over $R$ is Noetherian.
\end{theorem}

\subsection{Representable and induced $\fig$-modules}

In this section, we define a class of ``induced'' modules with highly controlled combinatorial behaviour. 
These modules play an essential role in Section \ref{figsharp}. The theory developed here, and our notation, closely mirrors Church--Ellenberg--Farb \cite[Section 2.2]{CEF} in the case of $\FI$-modules. The definitions below for general $\fig$-modules appear in Casto \cite[Section 2]{casto}, and Ramos \cite[Definition 2.4]{Ramos}.

\begin{definition} \label{DefnM(d)} Fix $d \geq 0$ and a commutative ring $R$. We denote by $M(d)$ the \emph{representable} or \emph{free} $\fig$-module over $R$, defined to be
$$M(d)_n = R[\Hom_{\fig}(\dX, \nX)]$$
with the action of $\fig$ morphisms by post composition, 
\begin{align*} f_*:  R[\Hom_{\fig}(\dX, \mX)] & \longrightarrow  R[\Hom_{\fig}(\dX, \nX)] \\ 
h & \longmapsto f \circ h
\end{align*} 
for $f \in \Hom_{\fig}(\mX, \nX),  h \in \Hom_{\fig}(\dX, \mX)$. 
\end{definition}

 We observe  that, for each $n$, $M(d)_n$ admits an action of $G\wr S_d = \End_{\fig}(\dX)$ on the right by precomposition. With this observation we can make the following definition. 

\begin{definition} \label{DefnInduced}
    Let $R$ be a commutative ring. Suppose that $T$ is a $R[G\wr S_d]-$module. Define the \emph{induced $\fig$-module} $M(T)$ by
    $$M(T)_n=M(d)_n\otimes_{R[G\wr S_d]}T$$
  and the $\fig$ morphisms act on the left tensor factor $M(d)$. More generally, we call an $\fig$-module an \emph{induced $\fig$-module} if it is a direct sum of $\fig$-modules of the form $M(T)$. 
\end{definition}

\begin{remark}
Sam and Snowden \cite[Section 1.1]{SS-Grobner} refer to modules of the form $M(d)$ as \emph{principal projective} modules. Gadish \cite[Definition 3.1]{Gadish-FItype} calls the modules $M(T)$ \emph{induction $\fig$-modules}. Ramos \cite[Definition 2.4]{Ramos} refers to sums of modules of the form $M(T)$ as \emph{relatively projective} modules, and Gadish \cite[Section 3]{Gadish-FItype} calls all such modules \emph{free}. 
\end{remark}

We collect some facts about the $\fig$-modules $M(d)$ and $M(T)$. These are direct analogues of the corresponding results for $\FI$-modules (see, for example, Church--Ellenberg--Farb \cite[Section 2.2]{CEF}), and we omit the details of the proofs. Recall the definition of the external tensor product $\boxtimes$ from Notation \ref{ExternalInternalTensor}. 

\begin{prop} \label{M(m)Facts} Let $R$ be a commutative ring. Fix a natural number $d \geq 0$, and let $T$ be a $G \wr S_d$-representation over $R$. 
    \begin{enumerate}[(i)] 
    \item $M(d)_n=0$ if $n<d$.
    \item $M(d)_d \cong R[G \wr S_d]$, the regular representation of $G \wr S_d$. 
    \item By an orbit-stabilizer argument, $M(d)_n$ is (as a  $G \wr S_n$-representation) the permutation representation on the cosets $$\End_{\fig}(\nX) / \mathrm{stab}(\iota_{d,n}) \; \; \cong \; \;  (G\wr S_n) / (G \wr S_{n-d})$$ 
    by Proposition \ref{PropDescribeStabilizer}. Consequently, there are isomorphisms of $G \wr S_n$-representations
    $$M(d)_n \cong \Ind_{G \wr S_{n-d}}^{G\wr S_n} R$$ where $R$ has trivial $G \wr S_{n-d}$ action. 
   \item $M(d)$ is finitely generated in degree $\leq d$ as an $\fig$-module by the identity morphism $id$ in $\End_{\fig}(\dX) \subseteq M(d)_d$. 
   \item Given any $\fig$-module $V$ and element $v \in V_d$,  the assignment $id \mapsto v$ extends uniquely to a map of $\fig$-modules $M(d) \to V$, 
   \begin{align*}
   M(d)_n & \longrightarrow V_n \\
   f & \longmapsto f_*(v) 
       \end{align*} 
 for $f \in \Hom_{\fig}(\dX, \nX)$. 
 \item \label{PropM(m)FG} Let $V$ be a $\fig$-module and $S \subseteq \coprod_n V_n$ a subset with $s \in V_{d_s}$. Then $V$ is generated by $S$ if and only if the corresponding map of $\fig$-modules $$\bigoplus_{s \in S}  M(d_s) \to V$$ 
 is surjective. In particular, $V$ is finitely generated if and only if it admits a surjective map of $\fig$-modules of the form $$\bigoplus_{1\leq i \leq N}  M(d_i) \to V$$ and generated in degree $\leq d$ if and only if it admits a surjective map of $\fig$-modules of the form $$\bigoplus_{0\leq m\leq d}  M(m)^{\oplus c_m} \to V.$$
    \item $M(d) \cong M(R[G \wr S_d])$.
    \item $M(T)_n \cong 0$ if $n<d$. 
    \item $M(T)_d \cong T$ as a $G\wr S_d$-representation. 
    \item \label{M(T)Induced} There are isomorphisms  $$M(T)_n\cong \Ind_{(G\wr S_d)\times (G \wr S_{n-d})}^{G \wr S_n}(T\boxtimes R) \qquad \text{as a $G \wr S_n$-representation}.$$ 
    Here, $(G\wr S_d)\times (G \wr S_{n-d})$ is the subgroup of $G \wr S_n$ that stabilizes the partition of $\nX=([n] \times X)$ into $\{1, \dots, d\} \times X$ and $\{d+1, \dots, n\} \times X$. Again $R$ has trivial $G \wr S_{n-d}$ action. 

    \item Since  $$\left( \frac{|G \wr S_n|}{|(G\wr S_d)\times (G \wr S_{n-d})|} \right)\ = { n \choose d}$$ then in particular, 
    $$M(T)_n \cong \bigoplus_{n \choose d} T \qquad \text{as an $R$-module}.$$
    Notably, if $T$ is generated as an $R$-module by a set of $N$ generators, then $M(T)_n$ has a $
    \left({n \choose d}N\right)$-element generating set. For $N\neq 0$ this is a polynomial in $n$ of degree $d$.  
    \item Let $R$ be a domain and $T$ be a $(G\wr S_d)$-module with finite rank over $R$. Then, as an $R$-module, $M(T)_n$ has rank 
   $$ \text{rank}_R(M(T)_n) = {n \choose d}\text{rank}_R(T).$$ Notably, when $T$ has positive rank, the rank of $M(T)_n$ is a polynomial in $n$ of degree $d$.  
    \item $M(T)$ is generated as an $\fig$-module in degree $\leq d$ by $M(T)_d \cong T$. 
    \item Given an $\fig$-module $V$, any $G \wr S_d$-equivariant map $T \to V_d$ extends uniquely to a map of $\fig$-modules $M(T) \to V$.  
    \end{enumerate} 
\end{prop}

\begin{remark} We call an $\fig$-module projective if it is a projective object in the category of $\fig$-modules. As with $\FI$-modules \cite[Remark 2.2.A]{CEF}, the projective $\fig$-modules are precisely the direct sums of induced modules $M(T)$ for $T$ a projective $R[G\wr S_d]$-module (Casto \cite[Section 2.2]{casto}, Ramos \cite[Proposition 2.13]{Ramos}). 
\end{remark}

\section{Representation stability} \label{SectionRepStability}

Church--Ellenberg--Farb proved that, if $V$ is a finitely generated $\FI$--module over $\mathbb{Q}$, for all $n$ sufficiently large the decomposition of $V_n$ into irreducible $S_n$--representations stabilizes in a strong sense \cite[Theorem 1.13]{CEF}. Gan--Li generalized this result to $\fig$-modules \cite[Theorem 1.12]{GL}. In this section we summarize their results. 

Let $G$ be a finite group. Throughout the section, assume that $\mathbb{K}$ is a subfield of $\C$ over which all complex irreducible representations of $G$ are defined. When  $G$ is a finite symmetric group, we may take $\mathbb{K}=\Q$. 

\begin{notation}
For a finite group $G$, let $\Irr(G)=\{\chi_1, \dots, \chi_r\}$ denote the set of its $r$ distinct isomorphism classes of irreducible representations over $\mathbb{K}$. Let $\chi_1 \in \Irr(G)$ denote the 1-dimensional trivial representation.  
\end{notation}

To make sense of what it means for the irreducible constituents of a sequence $V_n$ of $G \wr S_n$-representations to stabilize, we need to associate to each irreducible $G \wr S_n$-representation an irreducible representation of $G \wr S_{n+1}$, of $G \wr S_{n+2}$, of $G \wr S_{n+3}$, \ldots. The following definition gives a uniform notation for these ``stable" sequences of irreducible representations.  Recall from Corollary \ref{IrrepsGwreathSn}  that the irreducible representations of $G \wr S_n$ are classified by $n$-multipartitions $\underline{\lambda}$ with $r$ components, that is, partition-valued functions
    $$\underline{\lambda} : \{\chi_1, \ldots, \chi_r\} \to  \{\text{ partitions of nonnegative numbers }\} \qquad \text{ satisfying } \sum_{i=1}^r |\underline{\lambda} (\chi_i)|=n.$$ 

    For a general multipartition $\underline{\lambda}$ with $r$ components, we write $|\underline{\lambda}|$ for the sum $\sum_{i=1}^r |\underline{\lambda} (\chi_i)|$. 
\begin{definition}[Gan--Li {\cite[Section 1.4]{GL}}]\label{new partition}
    Let $\underline{\lambda}$ be any partition-valued function on $\Irr(G)$. Let $\underline{\lambda}(\chi_1)_1$ denote the largest part of the partition $\underline{\lambda}(\chi_1)$; by convention let $\underline{\lambda}(\chi_1)_1=0$ if $\underline{\lambda}(\chi_1)$ is the empty partition.  Let $n \in \Z_{\geq 0}$ such that $n \geq |\underline{\lambda}| + \underline{\lambda}(\chi_1)_1$.  Define $$\underline{\lambda}[n](\chi)=\begin{cases}
(n-|\underline{\lambda}|, \ \underline{\lambda}(\chi_1)) & \chi=\chi_1 \\
\underline{\lambda}(\chi) & \chi \neq \chi_1
\end{cases}$$
Let  $L(\underline{\lambda})_n$ denote the irreducible $\mathbb{K}[G \wr S_n]$-module corresponding to the multipartition $\underline{\lambda}[n]$ when $n \geq |\underline{\lambda}| + \underline{\lambda}(\chi_1)_1$, and let $L(\underline{\lambda})_n =0$ for $n < |\underline{\lambda}| + \underline{\lambda}(\chi_1)_1$. 
\end{definition}

\begin{remark}\label{young diagram remark} We may visualize a multipartition $\underline{\lambda}$ as an $r$-tuple of Young diagrams. The multipartition  $\underline{\lambda}[n]$, then, is constructed from $\underline{\lambda}$ by adding a row of length $n-|\underline{\lambda}|$ to the top of the Young diagram $\underline{\lambda}(\chi_1)$ in the first component. As $n$ increases, the tuples $\underline{\lambda}[n]$ and $\underline{\lambda}[n+1]$ differ only in the length of the top row of the first Young diagram. \\

We illustrate this pattern in the case $G= \Z/3\Z$ (so $r=3$) and the multipartition $\underline{\lambda} = \Big( \;\Y{1}  \; , \; \Y{1,1} \; , \; \Y{2} \Big)$. In this case, we have
\begin{align*}
\underline{\lambda}[n] &\text{ undefined for $n<6$} \\
\underline{\lambda}[6] & = \Big( \;\Y{1,1}  \; , \; \Y{1,1}\;, \; \Y{2} \Big)\\ 
\underline{\lambda}[7] & = \Big( \;\Y{2,1}  \; , \; \Y{1,1}\;, \; \Y{2} \Big)\\ 
\underline{\lambda}[8] & = \Big( \;\Y{3,1}  \; , \; \Y{1,1}\;, \; \Y{2} \Big)\\ 
\underline{\lambda}[9] & = \Big( \;\Y{4,1}  \; , \; \Y{1,1}\;, \; \Y{2} \Big) \\ 
\vdots
\end{align*}
\end{remark}

 Using this notation, the following definition states the criteria for stability for a consistent sequence of  $G \wr S_n$-representations. This generalizes the notion of representation stability stated for symmetric groups in the introduction in Definition \ref{DefnRepStableSn}.

\begin{definition}[{Church--Farb \cite[Definition 1.1]{ChurchFarb}, Gan--Li \cite[ Definition 1.10]{GL}}] \label{RepStabDef} Let $\{\varphi_n: V_n \to V_{n+1}\}$ be a consistent sequence of $\mathbb{K}[G \wr S_n]$-modules $V_n$ in the sense of Definition \ref{DefnConsistentSequence}. 
This sequence is \emph{representation stable} if there exists $N > 0$ such that for each $n \geq N$, the following three conditions hold:
    \begin{enumerate}[(i)]
        \item $\textbf{Injectivity:}$ The map $\varphi_n : V_n \to V_{n+1}$ is injective.
        \item $\textbf{Surjectivity:}$ The span of the $G \wr S_{n+1}$-orbit of $\varphi_n(V_n)$ is all of $V_{n+1}$.
        \item $\textbf{Multiplicities:}$ There is a decomposition
        $$V_n= \bigoplus_{\underline{\lambda}} L(\underline{\lambda})_n^{\oplus c(\underline{\lambda})}$$
        for each $n \geq N$ and where the multiplicities $0 \leq c(\underline{\lambda}) \leq \infty$ do not depend on $n$ and with $|\underline{\lambda}| < N$ for each $\underline{\lambda}$ with positive multiplicity $c(\underline{\lambda})$. 
    \end{enumerate}
\end{definition}

Condition (iii) of Definition $\ref{RepStabDef}$ implies that, for all $n$ large enough,  we can express the decomposition of $V_n$ into irreducible subrepresentations uniformly in $n$ using finitely many multipartitions $\underline{\lambda}$. By Remark $\ref{young diagram remark}$, for all $n$ large enough, the multipartitions $\underline{\lambda}[n]$ indexing the irreducible constituents of $V_n$ only grow in $n$ in the first row of the first component.

Following the work of Church--Ellenberg--Farb on $\FI$-modules, Gan and Li \cite{GL} related finite generation of $\fig$-modules $V$ to representation stability for sequences $(\iota_{n,n+1})_*: V_n \to V_{n+1}$. 

\begin{theorem}[{Church--Ellenberg--Farb \cite[Theorem 1.13]{CEF}, Gan--Li \cite[Theorem 1.12]{GL}}]\label{representation stable}
    Suppose that $\mathbb{K}$ is a splitting field for $G$ of characteristic 0. Let $V$ be an $\fig$-module over $\mathbb{K}$. Then $V$ is finitely generated as an $\fig$-module if and only if the consistent sequence $\{(\iota_{n,n+1})_*: V_n \to V_{n+1}\}$ is a representation stable sequence of $\mathbb{K}[G \wr S_n]$-modules with $\dim_R V_n < \infty$ for each $n$.
\end{theorem}

\section{$\fig\sharp$-modules} \label{figsharp}

Church--Ellenberg--Farb \cite[Section 4.1]{CEF} study a class of $\FI$-modules $V$ with additional structure which, they proved, places even stronger constraints on the combinatorics of the representations $V_n$ and their characters. Roughly speaking, these are the $\FI$-modules for which the $\FI$ morphisms act by split injections, and admit one-sided inverses (satisfying appropriate compatibility conditions). More precisely, these are $\FI$-modules for which the action of $\FI$ extends to an action of a larger category---denoted by $\FI\sharp$ in Church--Ellenberg--Farb---containing as subcategories both $\FI$ and its opposite category. 

Ramos \cite[Section 2.3]{Ramos} and Casto \cite[Section 2.5]{casto} generalize the theory of $\FI\sharp$-modules to the setting of $\fig$. In this section we review the main results. We use a different characterization of the category $\fig\sharp$ that was suggested to the fourth author by Peter May during her thesis work. 

\begin{definition} \label{DefnFIGsharp} Let $G$ be a finite group. Let $X$ be a faithful transitive $G$-set as in Definition \ref{DefnFIG}. For a finite set $A$, let $(A \times X)_{\bp}$ denote the disjoint union of the set $A \times X$ with a basepoint that we denote by $\bp$. 

Define the category $\fig\sharp$ as follows. The objects are the based sets $(A \times X)_{\bp}$ with basepoint $\bp$  for each finite set $A$. The morphisms are based maps of sets $f: (A \times X)_{\bp} \to (B \times X)_{\bp} $ satisfying the following conditions.
\begin{itemize}
    \item $f(\bp) = \bp$
 \item For any point $(b,x)$ in the codomain distinct from the basepoint, its preimage contains at most 1 element of $A \times X$. We say $f$ is ``injective away from the basepoint.'' 
 \item For each $a \in A$, the image of $(\{a\} \times X)$ is either the basepoint $\{\bp\}$ or the set $(\{b\} \times X)$ for some $b \in B$. In the latter case, the map $$f|_{(\{a\} \times X)} : (\{a\} \times X) \to (\{b\} \times X)$$ is $G$-equivariant. 
\end{itemize}

For brevity we denote the object $([n]\times X)_{\bp}$ by $\nXbp$, for all $n \in \Z_{\geq 0}$.  

\end{definition} 

$\fig$ embeds in $\fig\sh$ as the wide subcategory of injective morphisms. Each injective morphism in $\fig\sh$ admits a left inverse in $\fig\sh$ that maps the complement of its image to the basepoint. These left inverses form a subcategory isomorphic to $\fig^{op}$. The subcategories $\fig, \fig^{op} \subseteq \fig\sh$ together generate $\fig\sh$.

\begin{remark} It is not too difficult to verify that the definition of $\fig\sh$ given in Definition \ref{DefnFIGsharp} is equivalent to the definition used by Casto \cite[Section 2.5]{casto}. The objects are indexed by finite sets. A morphism from $A$ to $B$ is a pair $(Z, f)$, where $Z \subseteq A$ is a subset, and $f: Z \to B$ is an $\fig$ morphism. We may identify $(Z, f)$ with the morphism (in the category of Definition \ref{DefnFIGsharp}) that is defined by $f$ on $Z \times X$ and that maps $(A\setminus Z) \times X$ to the basepoint. 
\end{remark}

As in the case of $\FI$-modules \cite[Example 4.1.4]{CEF}, the action of $\fig$ on $M(d)$ and $M(T)$ extends to an action of $\fig\sh$. For $f \in \Hom_{\fig\sh}(\mXbp, \nXbp),  h \in \Hom_{\fig}(\dX, \mX) \subseteq \Hom_{\fig\sh}(\dXbp, \mXbp)$, we define

\begin{align*} f_*:  R[\Hom_{\fig}(\dX, \mX)] & \longrightarrow  R[\Hom_{\fig}(\dX, \nX)] \\ 
h & \longmapsto \left\{ \begin{array}{ll} (f \circ h)\big\vert_{\dX}, & \text{ if  $\bp \notin (f\circ h)(\dX)$} \\ 0, & \text{ otherwise} \end{array} \right.
\end{align*} 
This in turn defines an action of $\fig\sharp$ on $M(T)$ for any $d$ and $(G \wr S_d)$-representation $T$. 

\begin{definition} For a commutative ring $R$, an $\fig\sh$-module is a functor $V$ from $\fig\sh$ to the category of left $R$-modules. A map of $\fig\sh$-modules over $R$ is a natural transformation of functors. We define kernels, images, etc, of $\fig\sh$-module maps pointwise, as with $\fig$-modules. Any $\fig\sh$-module map will also be a map of $\fig$-modules under the restriction of the actions to $\fig \subseteq \fig\sh$.
\end{definition}

\subsection{The structure theorem for $\fig\sh$-modules} 

Church--Ellenberg--Farb \cite[Theorem 4.1.5]{CEF} give a complete classification of $\FI\sh$-modules. Their argument adapts readily to $\fig\sh$-modules; see Ramos \cite[Theorem 2.19]{Ramos} or Casto \cite[Proposition 2.7]{casto}. 

\begin{theorem}[{Church--Ellenberg--Farb \cite[Theorem 4.1.5]{CEF}; Casto \cite[Proposition 2.7]{casto}}]\label{structure theorem figsharp} 
    Any $\fig\sh$-module, up to isomorphism of $\fig\sh$-modules, has the form $\bigoplus_{d\geq 0}M(T_d)$ for $T_d$ a $G \wr S_d$-representation (possibly 0).
\end{theorem}
Theorem \ref{structure theorem figsharp} states that the underlying $\fig$-module of an $\fig\sharp$-module is an induced $\fig$-module in the sense of Definition \ref{DefnInduced}. Concretely, in light of the description of $M(T)$ in Proposition \ref{M(m)Facts}, Theorem \ref{structure theorem figsharp} implies that any $\fig\sh$-module  has the following structure. 

\begin{corollary} \label{StructureTheoremCorollary}
     Let $V$ be a $\fig\sh$-module over a commutative ring $R$. 
Then there is some sequence of $G \wr S_d$-representations $T_d$ such that, as $G \wr S_n$-representations, 
$$V_n \cong \bigoplus_{d \geq 0} \Ind_{(G\wr S_d)\times (G \wr S_{n-d})}^{G \wr S_n}(T_d\boxtimes R) \qquad \text{for all $n$,}$$ where the summand  $\Ind_{(G\wr S_d)\times (G \wr S_{n-d})}^{G \wr S_n}(T_d\boxtimes R)$ is understood to be zero whenever $n<d$. 

The underlying $\fig$-module of $V$ is finitely generated in degree $\leq m$ if and only if $T_d=0$ for all $d > m$. 

If $R$ is a domain and each representation $T_d$ has finite rank as an $R$-module, then
$$ \text{rank}_R(V_n) = \sum_{d \geq 0} {n \choose d} \text{rank}_R(T_d) . $$ 
If some representation $T_d$ has infinite rank, then $V_d$ will have infinite rank. 
\end{corollary} 

As observed implicitly by Gan--Li \cite[Section 8.1]{GL} and Casto \cite[Corollary 2.8]{casto}, Pieri's rule for wreath products (stated in Corollary \ref{PieriRuleWreath}) therefore implies that finitely generated $\fig\sh$-modules are representation stable in the sense of Definition \ref{RepStabDef}, with stable range depending only on the degree of generation. In the following corollary, we collect this and other consequences of the structure theorem for $\fig\sh$-modules. 

\begin{corollary} \label{stabilization range} \label{characterization finitely generated fig} Let $V$ be an $\fig\sh$-module over a commutative ring $R$. The following hold. 
\begin{enumerate}[(i)]
\item Suppose, as an $\fig$-module, $V$ is generated in degree $\leq d$. Then $V \cong \bigoplus_{0 \leq m  \leq d} M(T_m)$ for some $G \wr S_m$ representations $T_m$ for $0 \leq m \leq d$. 
\item The $\fig$-module $V \cong \bigoplus_m M(T_m)$ is finitely generated  if and only if $\bigoplus_m T_m$ is finitely generated as an $R$-module, that is, each $R$-module $T_m$ is finitely generated, and nonzero for only finitely many $m$. 
\item Let $R = \mathbb{K}$ be a splitting field  for $G$ of characteristic 0. Suppose, as an $\fig$-module, $V$ is finitely generated in degree $\leq d$. Then $V$ is representation stable in the sense of Definition \ref{RepStabDef}, stabilizing for $n \geq 2d$. 
\item Let $R$ be a domain. Suppose $V$ is finitely generated as an $\fig$-module in degree $\leq d$. Then the sequence rank$_{R}(V_n)$ is equal to a polynomial in $n$, and the polynomial has degree at most $d$. 
\item  Conversely, suppose $V_n$ is a free $R$-module for all $n$, and suppose rank$_{R}(V_n)$ is a polynomial in $n$ of degree $\leq d$.  Then $V \cong \bigoplus_{0 \leq m \leq d} M(T_m)$. Consequently, as an $\fig$-module, $V$ is finitely generated in degree $\leq d$. If rank$_{R}(V_n)$ is a polynomial of degree exactly $d$, then $T_d \neq 0$, and the bound on degree of generation is sharp, in the sense that $V$ is not generated in degree $\leq (d-1)$. 
\end{enumerate}

\end{corollary}

\section{Character polynomials for $\fig$-modules} \label{SectionCharacterPolynomials}

Casto shows that, as with $\FI\sh$-modules (Church--Ellenberg--Farb \cite[Theorem 4.1.7]{CEF}), the characters of $\fig\sh$-modules are highly constrained and admit a description that is uniform in $n$.  

Recall from Section \ref{SectionWreathConjClasses} that the conjugacy classes of $G \wr S_n$ are classified by the cycle type of a permutation in $S_n$, along with (associated to each unordered cycle) a conjugacy class of an element of $G$. 

\begin{notation} Let $G$ be a finite group, and let $C_1, \dots, C_r$ be a complete list of its distinct conjugacy classes. Fix an integer $d \geq 1$. For any $n$, let $X_d^{C_i}$ denote the class function on $G \wr S_n$ that outputs the number of $d$-cycles in its cycle type labelled by the conjugacy class $C_i$. We call the functions $X_d^{C_i}$ the \emph{coloured cycle-counting functions}, as they count the number of cycles in the cycle type of an element of $G\wr S_n$ that are ``coloured" by a specified conjugacy class of $G$. 
\end{notation}
In the notation of James--Kerber (see Definition \ref{DefnWreathConjClasses}), the function $X_d^{C_i}$ is the function $a_{id}$. \\

As Casto calculates explicitly \cite[Theorem 2.2]{casto}, the formula for $M(T)$ as an induced representation (Theorem \ref{M(m)Facts} part (\ref{M(T)Induced})) allows us to express the character of $M(T)_n$ in terms of the coloured cycle-counting functions. The result is that the sequence of characters of $M(T)_n$ are given by a unique polynomial in the coloured cycle-counting functions, independent of $n$. 

The coloured cycle-counting functions $X_d^{C_i}$ (viewed as functions on the set $\coprod_n G\wr S_n$) are algebraically independent, and we consider the polynomial ring $$\mathbb{K}[\; X_d^{C_i} \; | \; d \geq 1, \; \text{$C_i$ a conjugacy class of $G$ }]$$ as a graded ring for which the element $X_d^{C_i}$ has degree $d$.

\begin{prop}[{\cite[proof of Theorem 2.2]{casto}}] \label{CharacterM(T)} Let $G$ be a finite group with conjugacy classes $C_1, \dots, C_r$. Let $\mathbb{K}$ be a splitting field for $G$ of characteristic zero.    Let $T$ be the $G\wr S_d$-representation with character $\chi^T$. Then, for every $n$, the $G \wr S_n$-representation $M(T)_n$ has character equal to the degree-$d$ character polynomial
$$ \chi^{M(T)_n} = \sum_{\substack{\text{conjugacy classes} \\ \text{$\sigma$ of $G \wr S_d$}}} \chi^T(\sigma) \left( \prod_{i} \prod_{1 \leq m \leq d} {X^{C_i}_m \choose X^{C_i}_m(\sigma)  }\right).$$
\end{prop}

In light of the structure theorem for $\fig\sh$-modules and its consequences (see Theorem \ref{structure theorem figsharp} and Corollary \ref{StructureTheoremCorollary}), Proposition \ref{CharacterM(T)} has the following corollary. 

\begin{corollary} \label{CorCharacterPolynomial}
     Let $V$ be an $\fig\sh$-module over a splitting field $\mathbb{K}$ of $G$ of characteristic zero. Suppose $\dim_{\mathbb{K}}(V_n)$ is finite for each $n$, and $V$ is generated as an $\fig$-module in degree $\leq m$. Then the sequence of characters $\chi^{V_n}$ is given by a single polynomial
$$P^V \in \mathbb{K}[\; X_d^{C_i} \; | \; d \geq 1, \; \text{$C_i$ a conjugacy class of $G$ }],$$ independent of $n$, in the coloured cycle-counting functions $X_d^{C_i}$, and $P^V$ has degree $\leq m$. 
\end{corollary}

Casto \cite[Theorem 2.2]{casto}, generalizing Church--Ellenberg--Farb, in fact proves that the characters of all finitely generated $\fig$-modules are equal to a character polynomial for $n$ sufficiently large, with the stable range depending on an invariant of $V$ called the \emph{relation degree}.

\section{Representation stability for vertical configuration spaces} \label{VConfSection} 

Recall the vertical configuration spaces $\widetilde{\cV}_{K}(\R^{p,q})$ of points in $\R^{p+q}$ defined in Definition \ref{DefnVConf}.  In this section, we prove the main theorems of our paper: representation stability for the (co)homology of these spaces, and classical (co)homological stability for the unordered vertical configuration spaces. 

In Section \ref{BK-Calculation}, we summarize the computation of the cohomology groups $H^{*}\left(\widetilde{\cV}_{K}(\R^{p,q})\right)$ due to Bianchi--Kranhold \cite{Vconf}. We use this calculation in Section \ref{BettiPolyGrowth} to prove that, in each cohomological degree, the ranks of these free abelian groups grow polynomially. Then, in Section \ref{CohomologyFunctorial}, we verify that these cohomology groups assemble to form an $\fig$-module. This enables us to prove our main theorem, Theorem \ref{mainThm}. Finally, in Section \ref{TransMapSection}, we review the theory of the transfer homomorphism of a regular covering space, and use it to deduce classical (co)homological stability for the unordered vertical configuration spaces. 

\begin{remark}
 In his beautiful work \cite{Gadish-SubspaceArrangements}, Gadish gave broad general topological criteria for a family of complements of linear subspace arrangements to satisfy (co)homological representation stability. We believe that these results would give an alternate approach to our main theorem, which does not rely on the calculations done by Bianchi--Kranhold. 
\end{remark}

\subsection{Bianchi--Kranhold's computation of the cohomology groups $H^{*}\left(\widetilde{\cV}_{K}(\R^{p,q})\right)$} \label{BK-Calculation}

Bianchi--Kranhold index a basis for the integral cohomology groups $H^{*}\left(\widetilde{\cV}_{K}(\R^{p,q})\right)$ using combinatorial objects they call \emph{ray partitions} \cite[Section 3.1 and Theorem 3.7]{Vconf}. We now review their definitions. \\

Throughout this subsection, we fix a partition $K=(k_1,\dots, k_r)$ with $k_i \geq 1$.
\begin{definition}[{\cite[Definition 3.1]{Vconf}}] \label{DefnTable}
    The table associated with the partition $K$ is the set
    $$T_K :=\{(i,j): 1 \leq i \leq r, 1 \leq j \leq k_i\}.$$
    We order $T_K$ lexicographically, which means we write $(i,j) < (i',j')$ if either $i<i'$, or $i=i'$ and $j<j'$.
\end{definition}
For the following two definitions, let $Q$ be a partition of $T_K$ into non-empty subsets $Q_1,\dots, Q_\ell$.
\begin{definition}[\cite{Vconf}, Notation 3.2]
    Let $\ell(Q)=\ell$ denote the $\emph{length}$ of the partition. Note that $$1 \leq \ell(Q) \leq |K|.$$
\end{definition}
\begin{definition}[{\cite[Notation 3.2]{Vconf}}]
    Let $\sim$ be the equivalence relation on the index set $\{1, \dots, \ell\}$ generated by the identifications $$\beta \sim \beta' \quad  \iff \quad  \text{$ \exists \; (i,j),  (i,j') \in T_K$ with $j\neq j'$ such that $(i,j) \in Q_\beta$ and $(i,j') \in Q_{\beta'}$.}$$ We define the $\emph{agility}$ of $Q$, denoted by $a(Q)$, to be the number of equivalence classes under $\sim$. Note that $$1 \leq a(Q) \leq \min(\ell(Q), r).$$ 
\end{definition}

\begin{definition}[{\cite[Definition 3.3]{Vconf}}] \label{DefnRayPartition}
    A $\emph{ray partition}$ $Q$ of type $K$ is a partition of non-empty subsets $Q_1,\dots, Q_\ell$ of $T_K$, with a total order $<_{\beta}$ on each piece $Q_\beta$ (called a \emph{ray}), such that the following hold:
    \begin{enumerate}[(i)]
        \item if $<$ denotes the global order, i.e., the order in $T_K$, then the parts $Q_i$ are indexed according to their global order, $$\text{min}(Q_1,<) < \cdots < \text{min}(Q_\ell,<).$$
        \item for each $1 \leq \beta \leq \ell$,
        $$\text{min}(Q_\beta, <_{\beta}) = \text{min}(Q_\beta, <).$$
    \end{enumerate}
\end{definition}

Bianchi--Kranhold \cite[Theorem 3.7]{Vconf} provide the following explicit description of the cohomology groups of $\widetilde{\cV}_{K}(\R^{p,q})$.

\begin{theorem}[{\cite[Theorem 3.7]{Vconf}}]\label{cohomology heorem Bianchi}
    Let $p \geq 0$, $q \geq 1$ and $K=(k_1,\dots, k_r)$ with $k_i \geq 1$. The integral cohomology $H^{*}(\widetilde{\cV}_{K}(\R^{p,q}))$ is freely generated as an abelian group by classes $u_Q$ for each ray partition $Q$, and the cohomological degree of $u_Q$ is
    $$|u_Q| = p(r-a(Q)) + (q-1)(|K|-\ell(Q)).$$
\end{theorem}




\subsection{The ranks of the (co)homology groups grow polynomially in $n$}\label{BettiPolyGrowth}

Throughout this section, we let $k, n \in \Z_{>0}$ and define a partition $K = (k, k, \dots,  k)$ such that $|K|=nk$. We set $p \geq 0$ and $q \geq 2$, and write $\widetilde{\cV}_{n}^{k}(\R^{p,q})$ for $\widetilde{\cV}_{K}(\R^{p,q})$. 

In this section, we show that for a fixed  $k$ and $d$, the rank of $H^d(\widetilde{\cV}_{n}^{k}(\R^{p,q}))$ grows polynomially in $n$ as $n \rightarrow \infty$. According to Theorem $\ref{cohomology heorem Bianchi}$, the degree-$d$ cohomology group $H^d(\widetilde{\cV}_{n}^{k}(\R^{p,q}))$ has a basis indexed by ray partitions $Q$ that satisfy
\begin{equation}\label{ray partitions equation}
    d=p(n-a(Q)) + (q-1)(nk-\ell(Q)).
\end{equation}

We first collect two general combinatorial lemmas. We then observe how Equation \ref{ray partitions equation} constrains the possible lengths of the ray partitions in fixed cohomological degree $d$. We use the combinatorial lemmas to bound the number of possible ray partitions of a given length. We deduce that $\text{rank}(H^d(\widetilde{\cV}_{n}^{k}(\R^{p,q})))$ is bounded by a polynomial of degree at most $\left\lfloor \frac{2d}{q-1} \right\rfloor$ when $q \geq 2$.  

\begin{notation}
    Let $S(n,t)$ denote the Stirling number of the second kind, that is, the number of ways to partition a set of $n$ objects into $t$ nonempty unordered subsets. 
\end{notation}   

See (for example) \cite[Chapter 1.3]{Stirling} for the combinatorial theory of Stirling numbers. Crucially, they satisfy the following defining combinatorial identity. 

\begin{prop} \label{combinatorial properties stirling number} For $n,t \in \Z_{>0}$, $$ S(n, t) = tS(n-1,t) + S(n-1, t-1) \quad \text{for $n \geq 1$. }$$ 
\end{prop}

\begin{prop}\label{stirling number theorem}
   For fixed $t$, the sequence $\left(S(n,n-t)\right)_{n>t}$ is polynomial in $n$ of degree $2t$. 
\end{prop}
\begin{proof}
    We proceed by induction on $t$. When $t = 0$, $S(n, n) = 1$ for any $n > t$. 
    Suppose now for some fixed $t\geq 0$ that $(S(n,n-t))_{n > t}$ is polynomial in $n$ of degree $2t$. We want to show that the sequence $s_n:=S(n,n-(t+1))$ is polynomial in $n$ of degree $2t+2$.   Proposition \ref{combinatorial properties stirling number} states that, for $n \geq 1$ and $n-t-1 \geq 1$, 
  $$  S(n,n-t-1) = (n-t-1)S(n-1, n-t-1) + S(n-1,n-t-2). $$
  In other words, the differences $s_n - s_{n-1}$ satisfy
$$ s_n - s_{n-1} = (n-t-1)S(n-1, n-t-1).$$
Let $p(n) = (n-t-1)S(n-1, n-t-1)$. Since $S(n-1, n-t-1)$ is polynomial in of degree $2t$ by the induction hypothesis, $p(n)$ is polynomial of degree $2t+1$. Then $s_n = \sum_{i=t+2}^n p(i)$ is a polynomial of degree $2t+2$.
\end{proof}

We also make the following elementary observation: 
\begin{lemma} \label{LemmaFactorial}
    For nonnegative integers $a_1, a_2, \dots, a_p$, 
    $$ (a_1 + a_2 +  \dots + a_p)! \geq a_1 ! a_2! \dots a_p! .$$
\end{lemma}

The following lemma gives a bound on the lengths of a ray partition $Q$ associated to a cohomology generator in cohomological degree $d$.  

\begin{lemma} \label{LemmaLengthBound}
   Suppose $q \geq 2$.  Then
    $$\left(nk - \frac{d}{q-1}\right) \leq \ell(Q) \leq  n k.$$
In particular, ray partitions $Q$ that index a cohomology basis element in fixed cohomological degree $d$ admit at most $\left( 1+ \left\lfloor \frac{d}{q-1} \right\rfloor \right)$ possible distinct values of $\ell(Q)$, a number that does not depend on $n$. 
\end{lemma}

\begin{proof}
    The right hand side is clear by the definition of $\ell(Q)$. The left hand side follows from the inequality
    $$ d \quad = \quad (q-1)(nk - \ell(Q)) + p(n-a(Q))  \quad \geq \quad (q-1)(nk - \ell(Q))$$ 
    since the term $p(n-a(Q))$ is nonnegative. 
\end{proof}

We now bound the number of ray partitions of a given length. 

\begin{lemma} \label{NumRay}
Fix an integer $t$ with $0 \leq t \leq \left\lfloor \frac{d}{q-1} \right\rfloor$.  As $n$ tends to infinity,  the number of ray partitions $Q$ of length $\ell(Q) =  nk -t$ is polynomial in $n$ of degree at most $2t$.  
\end{lemma}

\begin{proof} There are $S(nk, nk -t)$ possible partitions of the table $T_K$ (Definition \ref{DefnTable}) of length $\ell = nk -t$. Denote the parts $Q_1, \dots, Q_{\ell}$ according to indexing convention (i) in the definition of a ray partition (Definition \ref{DefnRayPartition}). By Definition \ref{DefnRayPartition}, then, there is an associated ray partition 
 $Q$ for each choice of total order on each part $Q_i$, each subject to a specified minimal element. The number of such choices of total orderings is
 \begin{align*}
    &   (|Q_1| - 1)! (|Q_2| - 1)! \dots (|Q_{\ell}|-1)! \\ 
    & \leq \left( \sum_i (|Q_i| - 1) \right)!  \quad \text{ by Lemma \ref{LemmaFactorial} }  \\ 
    & = (n k - \ell) ! \\ 
    & = t! .
 \end{align*}
 Since $t!$ is constant in $n$, the result follows from Lemma \ref{stirling number theorem}. 
\end{proof}

We can now prove the main theorem of this section.
\begin{theorem}\label{polynomial bound on homology dimension}
    Let $p \geq 0$ and $q \geq 2$. For a fixed degree $d$, as $n$ tends to infinity, the rank of $H^d(\widetilde{\cV}_{n}^{k}(\R^{p,q});\Z)$ is polynomial  in $n$ of degree at most $\left\lfloor \frac{2d}{q-1}\right\rfloor$.
\end{theorem}

\begin{proof} The result follows from Lemmas \ref{LemmaLengthBound} and \ref{NumRay}. 
\end{proof}

Since integral cohomology groups $H^d(\widetilde{\cV}^k_n (\R^{p,q});\Z)$ are finitely generated and free for all $d$ (Theorem \ref{cohomology heorem Bianchi}), the universal coefficient theorems imply the following.

\begin{corollary}\label{bounding homology}
    Let $p \geq 0$ and $q \geq 2$. For a fixed degree $d$, as $n$ tends to infinity, $\text{rank}(H_d(\widetilde{\cV}_{n}^{k}(\R^{p,q}); \Z))$ is polynomial in $n$ of degree at most $\left\lfloor \frac{2d}{q-1}\right\rfloor$. 
\end{corollary}

\subsection{ The (co)homology groups of vertical configuration spaces form $\fig\sh$-modules} \label{CohomologyFunctorial}

Fix $p \geq 0$, $q \geq 2$, $k \in \Z^+$, and $K=(k,\dots, k)$ with $|K|=nk$. Let $\widetilde{\cV}_{n}^{k}(\R^{p,q})$ denote the ordered vertical configuration space of $n$ clusters with $k$ points in each.

 Let $G=S_k$.  Our goal is to verify that, for a fixed (co)homology degree $d$, the (co)homology groups of $\widetilde{\cV}^{k}_{n}(\R^{p,q})$ admit an $\fig\sh$-module structure. Our construction closely parallels Church--Ellenberg--Farb \cite[Proposition 6.4.2]{CEF}. We will view the objects of $\fig$ as sets of the form $A \times [k]$ for $A$ a finite set, as discussed in Remark \ref{RemarkDefnFIG}, and analogously view objects of $\fig\sh$ as based sets $(A\times [k]) \cup \{\bp\}$. 
 We first adapt some definitions introduced in \cite{CEF} to $\fig$-modules. 
\begin{definition}
    An \emph{$\fig$-space} is a covariant functor $X$ from $\fig$ to the category  $\text{Top}$ of topological spaces. A \emph{co-$\fig$-space} is a contravariant functor $X$ from  $\fig$ to  $\text{Top}$. \emph{$\fig\sh$-spaces} and \emph{co-$\fig\sh$-spaces} are defined analogously. 
\end{definition}
\begin{definition}\label{homotopy fig-space}
    A $\emph{homotopy $\fig$-space}$ is a functor $X$ from $\fig$ to hTop, the category of topological spaces and homotopy classes of continuous maps. $\emph{Homotopy $\fig\sh$-spaces}$ are defined analogously.  
\end{definition}
Since a homotopy class of maps induces a well-defined map on homology and cohomology, the sequences of singular homology and cohomology groups of a homotopy $\fig\sh-$space form $\fig\sh$-modules.

\begin{notation}
    Let $B$ be a finite set. Write
    $$T_B^k:=\{(b,j)| \ \text{$b\in B$ and $1\leq j\leq k$}\} = B \times [k]$$
\end{notation}    
In other words, $T_B^k$ is the object of $\fig$ associated to the finite set $B$. This notation is suggestive of Definition \ref{DefnTable}, as $T_{(k, \dots, k)}$ is the table $T^{k}_B$  for $B=[n]$. 

\begin{definition} \label{CategorifyingVer}
Let $\pi_p$ be the projection of $\R^{p,q}$ onto the first $p$ coordinates. Define the space
$$\text{Ver}_B^{k}(\R^{p,q}):=\{\rho \ | \ \rho:T_B^{k}\hookrightarrow \R^{p,q} \text{ satisfying } \pi_p(\rho(b,i))=\pi_p(\rho(b,j)) \text{ for } (b,i), (b,j) \in T_{B}^k\}.$$
\end{definition}
A choice of bijection of sets $B \cong [|B|]$ defines an $S_{k}\wr S_{|B|}$-equivariant homeomorphism $\text{Ver}_B^{k}(\R^{p,q}) \cong \widetilde{\cV}_{|B|}^{k}(\R^{p,q})$. The ``categorical" description of the vertical configuration spaces of Definition \ref{CategorifyingVer} enables us to directly define a contravariant action of the category $\fig$, with morphisms acting by precomposition. 

\begin{prop} \label{VerCoFI} $\text{Ver}_{\blacksquare}^{k}(\R^{p,q})$ is a co-$\fig$-space, where an  $\fig$ morphism $f: T^{k}_A \to T^{k}_B$ acts by
\begin{align*}
    f^*: \text{Ver}_{B}^{k}(\R^{p,q}) &\longrightarrow \text{Ver}_{A}^{k}(\R^{p,q})  \\ 
    \rho & \longmapsto \rho \circ f.
\end{align*}
    
\end{prop}

    To extend this co-$\fig$-action to an action of $\fig\sh$, we define stabilization maps that introduce a new cluster of points ``far away." 
\begin{definition} \label{DefnOrderedStabilization}
Fix $p \geq 0, q \geq 1$, and an ordered partition $(k_1, \dots, k_{n}, k_{n+1})$. Let $e_1, \dots e_{p+q}$ denote the standard basis for $\R^{p+q}$, and $pr_{e_1}$ the projection onto the first component. Define a map 
$$\tilde{\phi}: \tilde\cV_{(k_1, \dots, k_{n})}(\R^{p,q}) \to \tilde\cV_{(k_1, \dots, k_{n}, k_{n+1})}(\R^{p,q})$$ as follows. We map an element $(z_1,\dots,z_n) \in \tilde\cV_{(k_1, \dots, k_{n})}(\R^{p,q})$ with clusters $z_i =(z_i^1,\dots,z_i^{k_i})$ to the element $\tilde\phi(z_1,\dots,z_n) =(z_1,\dots,z_n, z_{n+1})$, where the new cluster $z_{n+1}=(z_{n+1}^1,\dots,z_{n+1}^{k_{n+1}})$ consists of the points 
$$ z_{n+1}^t = \left(\max_{\substack{1 \leq i \leq n \\ 1 \leq j \leq k_i}} |pr_{e_1}(z_i^j)| + 1 \right)e_1 + t\,e_{p+q}$$
as illustrated in Figure \ref{OrderedStablization}. 
For the partition $(k, \dots, k)$, we denote the stabilization map 
$$ \tilde{\phi}_n: \tilde\cV^{k}_{n}(\R^{p,q}) \longrightarrow \tilde\cV^{k}_{n+1}(\R^{p,q}).$$ 
\end{definition}

\begin{figure}[h!]  
\begin{center}
    \includegraphics{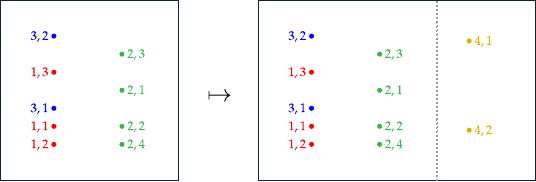}
    \caption{The stabilization map $\tilde\cV_{({\color{red}3},{\color{Green}2},{\color{blue}4})}(\R^{1,1}) \to \tilde\cV_{({\color{red}3},{\color{Green}2},{\color{blue}4}, {\color{brown}2})}(\R^{1,1})$. Figure adapted from Bianchi--Kranhold \cite[Figure 1]{Vconf}.} \label{OrderedStablization}       
        \end{center}
        \end{figure} 

        We observe that, because $q\geq 2$ by assumption,  if we permute the labels on the points in the new cluster $z_{n+1}$, the resulting map would be homotopic (continuously in $\tilde\cV^{k}_{n}(\R^{p,q})$) to $\tilde\phi_n$. Similarly, up to homotopy, we can place the cluster further away in the $e_1$ direction. 

                The stabilization maps $\tilde\phi_n$ are split via the natural forgetful maps $\tilde\phi^n$, 
        \begin{align*}
            \tilde\cV_{(k_1, \dots, k_{n}, k_{n+1})}(\R^{p,q})  & \longrightarrow \tilde\cV_{(k_1, \dots, k_{n})}(\R^{p,q}) &
            \tilde\phi^n:  \tilde\cV^{k}_{n+1}(\R^{p,q}) &\longrightarrow \tilde\cV^{k}_{n}(\R^{p,q}).\\ 
           (z_1,\dots,z_n, z_{n+1}) & \longmapsto (z_1,\dots,z_n)  
           &
            (z_1,\dots,z_n, z_{n+1}) & \longmapsto (z_1,\dots,z_n).
        \end{align*}

        \begin{prop} \label{VerHomotopyFI} $\text{Ver}_{\blacksquare}^{k}(\R^{p,q})$ is a  homotopy $\fig$-space, where the action of (the skeleton of)  $\fig$ is generated by the action of the groups $S_{k}\wr S_n$ and the stabilization maps $ \tilde{\phi}_n: \tilde\cV^{k}_{n}(\R^{p,q}) \to \tilde\cV^{k}_{n+1}(\R^{p,q})$ of Definition \ref{DefnOrderedStabilization}.    
\end{prop}

\begin{proof} By construction, the sequence $\Big(\tilde{\phi}_n: \tilde\cV^{k}_{n}(\R^{p,q}) \to \tilde\cV^{k}_{n+1}(\R^{p,q})\Big)_n$ is a consistent sequence in the sense of Definition \ref{DefnConsistentSequence}. By Proposition \ref{fig criterion}, it suffices to check that, for all $\omega \in \mathrm{stab}(\iota_{m,n}) \cong S_{k} \wr S_{n-m}$, there is equality between $\omega (\iota_{m,n})_* = \omega \tilde{\phi}_{n-1} \cdots \tilde{\phi}_m$ and 
$(\iota_{m,n})_* = \tilde{\phi}_{n-1} \cdots \tilde{\phi}_m$. These are not equal as functions of spaces---we have potentially permuted the labels on the points in the new clusters introduced ``far away"--but because $q \geq 2$, they are homotopic maps. This concludes the proof. 
\end{proof}
 
The co-$\fig$-space structure of Proposition \ref{VerCoFI} and the homotopy $\fig$-space structure of Proposition \ref{VerHomotopyFI} combine to give the spaces $\text{Ver}_{\blacksquare}^{k}(\R^{p,q})$ the structure of a homotopy $\fig\sh$-space. We can verify functoriality directly using the arguments from Church--Ellenberg--Farb \cite[Proof of Proposition 6.4.2]{CEF}.

\begin{theorem} \label{SharpStructure}
$\text{Ver}_{\blacksquare}^{k}(\R^{p,q})$ admits a homotopy $\fig\sh$-space structure with respect to the maps defined in Proposition \ref{VerCoFI} and Proposition \ref{VerHomotopyFI}. 
\end{theorem}

To conclude the following corollary, recall that $\fig\sharp$ is isomorphic to its own opposite category. 

\begin{corollary} \label{H^d(V)HasFIGStructure}
    Fix $d \geq 0$ and a commutative ring $R$. The sequence of degree-$d$ cohomology groups  $H^d(\widetilde{\cV}^{k}_{n}(\R^{p,q});R)$, and the sequence of degree-$d$ homology groups $H_d(\widetilde{\cV}^{k}_{n}(\R^{p,q}); R)$, form $\fig\sh$-modules.
\end{corollary}

\subsection{ The (co)homology  of vertical configuration space is representation stable} 

We now prove the main theorem. 

\begin{theorem} \label{mainThm}
     Let $p \geq 0$ and $q \geq 2$. Fix  $k \geq 1$ and (co)homological degree $d \geq 0$. 
\begin{enumerate}[(i)]
    \item \label{MainThmI} Let $R$ be a commutative ring.  The sequences $\left( H_d(\widetilde{\cV}^k_n (\R^{p,q});R) \right)_{n}$ and $\left( H^d(\widetilde{\cV}^k_n (\R^{p,q});R) \right)_{n }$ of $S_{k} \wr S_n$-representations are representation stable in the sense that they are $\FI_{S_k}\sh$-modules finitely generated in degree at most $ \displaystyle \left\lfloor \frac{2d}{q-1} \right\rfloor$. 
    \item \label{CorMultiplicityStability}    The decomposition of the rational $S_k \wr S_n$-representations $H^d(\widetilde{\cV}^k_n (\R^{p,q});\Q)$ into irreducible components $L(\underline{\lambda})_n$ stabilizes in $n$ in the sense of Definition \ref{RepStabDef}, with stable range $ \displaystyle n \geq \left\lfloor \frac{4d}{q-1} \right\rfloor$. The same statement holds for $H_d(\widetilde{\cV}^k_n (\R^{p,q});\Q)$.  
\end{enumerate}
\end{theorem}
\begin{proof}
    By Corollary \ref{H^d(V)HasFIGStructure}, for any fixed $d$,  the sequences of representations $\left( H^d(\widetilde{\cV}^k_n (\R^{p,q});R)\right)_{n\geq 1}$ and $\left( H_d(\widetilde{\cV}^k_n (\R^{p,q});R)\right)_{n\geq 1}$  each have the structure of an $\FI_{S_k}\sh-$module. Recall Theorem \ref{cohomology heorem Bianchi} that the integral cohomology groups $H^d(\widetilde{\cV}^k_n (\R^{p,q});\Z)$ are finitely-generated  free abelian groups in all degrees $d$, so by the universal coefficient theorem, $H^d(\widetilde{\cV}^k_n (\R^{p,q}); R)$ and $H_d(\widetilde{\cV}^k_n (\R^{p,q}); R)$ are also free with the same rank. By Theorem \ref{polynomial bound on homology dimension} and Corollary \ref{bounding homology}, their ranks are polynomial in $n$ of degree at most $\left\lfloor \frac{2d}{q-1} \right\rfloor$. The result now follows from Corollary $\ref{stabilization range}$.
\end{proof}

 Corollary \ref{StructureTheoremCorollary} and 
Corollary \ref{CorCharacterPolynomial} imply the following two consequences of representation stability.

\begin{corollary} \label{CorInducedReps}
    For $i=0, \dots \left\lfloor \frac{2d}{q-1} \right\rfloor$, there exist $S_{k} \wr S_i$-representations $U_i$ such that, as $S_{k} \wr S_n$-representations, 
$$H^d(\widetilde{\cV}^k_n (\R^{p,q});R) \cong \bigoplus_{i= 0}^{\left\lfloor \frac{2d}{q-1} \right\rfloor} \Ind_{(S_{k}\wr S_i)\times (S_{k} \wr S_{n-i})}^{S_{k} \wr S_n}(U_i\boxtimes R) \qquad \text{for all $n$.}$$ 
Here the representation $\Ind_{(S_{k}\wr S_i)\times (S_{k} \wr S_{n-i})}^{S_{k} \wr S_n}(U_i\boxtimes R)$ is understood to be zero whenever $n<i$. The same statement holds for $H_d(\widetilde{\cV}^k_n (\R^{p,q});R) $. 
\end{corollary}



\begin{corollary} \label{CorHomologyCharacterPolynomial} Fix $d \geq 0$. The sequence of characters of the $S_{k} \wr S_n$-representations $H^d(\widetilde{\cV}^k_n (\R^{p,q});\Q)$ can be expressed as a single character polynomial in the coloured cycle-counting functions
    $$P^{H^d(\widetilde{\cV}^k(\R^{p,q}))} \in \mathbb{Q}[\; X_r^{C_i} \; | \; r \geq 1, \; \text{$C_i$ a conjugacy class of $S_{k} $ }] $$ of degree at most  $\left\lfloor \frac{2d}{q-1} \right\rfloor$, independent of $n$. The same statement holds for the sequence of characters of  $H_d(\widetilde{\cV}^k_n (\R^{p,q});\Q)$. 
\end{corollary}

\begin{remark}
    The range in Theorem \ref{mainThm} (\ref{CorMultiplicityStability}) may not be sharp in general. In the case $p = 0$ and $k = 1$, which is the ordinary ordered configuration space $\Conf_n (\R^q)$, Theorem \ref{mainThm} gives us the stable range $n \geq \left\lfloor \frac{4d}{q-1} \right\rfloor $ for $q \geq 2$. In this case the cohomology groups are nonzero only in degrees $d$ divisible by $(q-1)$, and for the nonzero groups the sharp range is given by Hersh--Reiner \cite[Theorem 1.1]{HershReiner}:  
    \begin{equation*}
  n \geq \begin{cases}
    \frac{3d}{q-1} & \text{if $q\geq 3$ odd}\\
    \frac{3d}{q-1} + 1 & \text{if $q \geq 2$ even}. 
  \end{cases}
\end{equation*}

\end{remark}

\subsection{The transfer map and classical (co)homological stability } \label{TransMapSection}
In this section we give another proof of the stability of the rational (co)homology group of the unordered vertical configuration spaces $\cV_n^{k}(\R^{p,q})$. Our proof uses representation stability and the transfer map. 

Analogous to the stabilization maps we defined on the ordered vertical configuration spaces in Definition \ref{DefnOrderedStabilization}, we define stabilization maps $\phi_n: \cV_n^{k}(\R^{p,q}) \to \cV_{n+1}^{k}(\R^{p,q})$ on the unordered spaces by introducing a $(n+1)^{st}$ cluster of $k$ points ``far away," as in Figure \ref{UnorderedStablization}. See Latifi \cite[Section 4.1]{Latifi} or Bianchi--Kranhold \cite[Construction 4.1]{Vconf}.

\begin{figure}[h!]  
\begin{center}
    \includegraphics{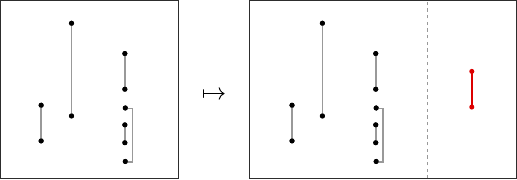}
    \caption{The stabilization map $\phi_5: \cV_5^{2}(\R^{1,1}) \to \cV_{6}^{2}(\R^{1,1})$. Figure adapted from Bianchi--Kranhold \cite[Figure 6]{Vconf}.} \label{UnorderedStablization}       
        \end{center}
        \end{figure} 

Latifi and Bianchi--Kranhold proved the following stability result with respect to the maps $\phi_n$. Their results generalize the work of Arnol'd \cite{Arnold} and others on homological stability for classical configuration spaces of points in $\R^s$, and more recent work of Palmer \cite{Palmer} on unlabelled clusters of points in $\R^s$ (without the ``verticality" condition). 

\begin{theorem}[{Latifi \cite[Theorem 4.1.4]{Latifi} and Bianchi--Kranhold \cite[Theorem 4.3]{Vconf}}]  For all $p \geq 0$ and $q \geq 1$ with $(p, q) \neq (0, 1)$, the induced maps
$$H_d(\cV^{k}_n(\R^{p,q})) \longrightarrow  H_d(\cV^{k}_{n+1}(\R^{p,q}))$$ 
are isomorphisms for all $n \geq 2d$.     
\end{theorem}
\noindent They note that stability fails when $(p,q)=(0,1)$. 

We give a new proof of rational (co)homological stability in the cases $p \geq 0$ and $q \geq 2$. Although our results only hold rationally, for $q \geq 3$ they improve the stable range.

To deduce (co)homological stability from Theorem \ref{mainThm}, we follow a standard line of reasoning in the representation stability literature (see, for example, Church--Ellenberg--Farb \cite[Corollary 6.3.4]{CEF}). We will, however, give a more detailed justification of the approach, as these details may not be easily found in the literature. We first review the theory of transfer maps of covering spaces. 

\subsubsection{The transfer map} 

Recall that, given a representation $U$ of a finite group $\cG$ over $\Q$, the ``averaging map" over the action of $\cG$
\begin{align*}
   \alpha_{\cG} \colon U &\longrightarrow U \\ 
    u & \longmapsto \frac{1}{|\cG|} \sum_{g \in \cG} g \cdot u,
\end{align*}
defines a projection from $U$ to its $\cG$-invariant subspace $U^{\cG}$, which factors through an isomorphism between the $\cG$-coinvariants and $\cG$-invariants $U_{\cG} \cong U^{\cG}$. 

We now recall some facts about transfer maps (see, for example, Hatcher \cite[Section 3.G]{Hatcher}). 

\begin{definition}
    Let $\widetilde{X}$ be an $N$-sheeted covering space $\pi: \widetilde{X} \rightarrow X$ of a space $X$ for some $N \in \mathbb{N}$. Then $\pi$ induces a map $\pi_{\#}:$ $C_d(\widetilde{X};\Q) \rightarrow C_d(X;\Q)$ on singular chains. Define a map 
    \begin{align*}
    \tau: C_d(X;\Q) & \longrightarrow C_d(\widetilde{X};\Q) \\
    [\sigma: \Delta^d \rightarrow X] & \longmapsto  \left[   \sum_{\text{lifts $\Tilde{\sigma_i}$ of ${\sigma}$}} \frac{(  \Tilde{\sigma_i}: \Delta^d \rightarrow \widetilde{X} )}{N} \right].
    \end{align*} 
\end{definition}
    
   The map $\tau$ is a chain map, so it induces \textit{transfer homomorphisms} $\tau_*: H_d(X; \Q) \rightarrow H_d(\widetilde{X}; \Q)$ on homology and $\tau^*: H^d(\widetilde{X}; \Q) \rightarrow H^d(X; \Q)$ on cohomology,  satisfying $\pi_* \tau_* = id$ and $\tau^* \pi^* = id$.   When $\pi$ is a regular cover with deck group $\cG$, then the map $\tau \pi_{\#}$ is the `averaging map' $\alpha_{\cG}$, 
\begin{align*}
 C_d(\widetilde{X}; \Q) \overset{\pi_{\#}}{\longrightarrow} & \; C_d(X;\Q)  \overset{\tau}{\longrightarrow}  C_d(\widetilde{X};\Q) \\ 
 \tilde{\sigma} \longmapsto & \;\pi_{\#}( \tilde{\sigma}) \longmapsto \frac{1}{|\cG|} \sum_{g \in \cG} g\cdot \tilde{\sigma}. 
\end{align*}
The induced map $\tau_* \pi_*$ is the averaging map on homology $H_d(\widetilde{X}; \Q)$, and the induced map $\pi^* \tau^*$ is the averaging map on cohomology $H^d(\widetilde{X}; \Q)$.  This implies the following (standard) result.

\begin{prop} \label{transferMap}
    Let $X$ be a topological space, $\widetilde{X}$ be the regular covering space with finite deck group $\cG$. 
Then the map $\pi^*: \  H^d(X; \Q) \rightarrow H^d(\widetilde{X};\Q)$ induces an isomorphism between $ H^d(X; \Q)$ and the $\cG$-invariant subspace of  $H^d(\widetilde{X};\Q)$ 
$$ H^d(X; \Q) \xrightarrow[]{\cong} H^d(\widetilde{X};\Q)^{\cG}.$$  The map $\tau^*:  H^d(\widetilde{X}; \Q) \to H^d(X;\Q) $ factors through an isomorphism between the $\cG$-coinvariants of  $H^d(\widetilde{X};\Q)$ and $H^d(X; \Q) $,  $$ H^d(\widetilde{X};\Q)_\cG \xrightarrow[]{\cong} H^d(X; \Q).  $$ Similarly, the maps $\pi_*$ and $\tau_*$, respectively, factor through isomorphisms
$$   H_d(\widetilde{X};\Q)_\cG \xrightarrow[]{\cong} H_d(X; \Q) \quad \text{and} \quad H_d(X; \Q) \xrightarrow[]{\cong} H_d(\widetilde{X};\Q)^{\cG}, \text{ respectively.} $$ 
\end{prop}

\subsubsection{Classical (co)homological stability as a consequence of representation stability} \label{SectionRepStabilityImpliesClassicalStability}

We will apply this theory to the situation where we have a sequence of covers $\widetilde{X}_n$ equipped with  maps $\tilde\varphi_n: \widetilde{X}_n \to \widetilde{X}_{n+1}$ and $\tilde\varphi^n: \widetilde{X}_{n+1} \to \widetilde{X}_{n}$ as shown 
\[\begin{tikzcd}[column sep=large]
	{C_d(\widetilde{X}_n;\mathbb{Q}) \; \;} & {\;\; C_d(\widetilde{X}_{n+1};\mathbb{Q})} \\
	{C_d({X}_n;\mathbb{Q}) \; \;} & {\; \; C_d({X}_{n+1};\mathbb{Q})}
	\arrow["{\tilde\varphi^n}"', from=1-2, to=1-1]
	\arrow["{\pi_{\#}}"', bend right, from=1-1, to=2-1]
	\arrow["{\pi_{\#}}"', bend right, from=1-2, to=2-2]
	\arrow["\tau"',  bend right, from=2-1, to=1-1]
	\arrow["{\varphi^n}", from=2-2, to=2-1]
	\arrow["\tau"', bend right, from=2-2, to=1-2]
\end{tikzcd} \qquad \qquad
\begin{tikzcd}[column sep=large]
	{C_d(\widetilde{X}_n;\mathbb{Q}) \; \;} & {\;\; C_d(\widetilde{X}_{n+1};\mathbb{Q})} \\
	{C_d({X}_n;\mathbb{Q}) \; \;} & {\; \; C_d({X}_{n+1};\mathbb{Q})}
	\arrow["{\tilde\varphi_n}", from=1-1, to=1-2]
	\arrow["{\pi_{\#}}"', bend right, from=1-1, to=2-1]
	\arrow["{\pi_{\#}}"', bend right, from=1-2, to=2-2]
	\arrow["\tau"',  bend right, from=2-1, to=1-1]
	\arrow["{\varphi_n}"', from=2-1, to=2-2]
	\arrow["\tau", bend right, from=2-2, to=1-2]
\end{tikzcd}\]
satisfying the relations $\varphi^n \pi_{\#} = \pi_{\#} \tilde{\varphi}^n$ and $\varphi_n \pi_{\#} = \pi_{\#} \tilde{\varphi}_n$. We assume that the induced maps 
$$\left((\tilde\varphi^n)^*: H^d(\widetilde{X}_n;\Q)\longrightarrow   H^d(\widetilde{X}_{n+1};\Q)\right)_n \text{ and } \left((\tilde\varphi_n)_*: H_d(\widetilde{X}_n;\Q)\longrightarrow  H_d(\widetilde{X}_{n+1};\Q)\right)_n$$ form consistent sequences of $G \wr S_n$-representations in the sense of Definition \ref{DefnConsistentSequence}, (or assume at a minimum that the composites 
$$H^d(\widetilde{X}_n;\Q)\xrightarrow{(\tilde\varphi^n)^*}   H^d(\widetilde{X}_{n+1};\Q) \twoheadrightarrow H^d(\widetilde{X}_{n+1};\Q)_{\cG_{n+1}} $$ $$  H_d(\widetilde{X}_n;\Q)\xrightarrow{(\tilde\varphi_n)_*}  H_d(\widetilde{X}_{n+1};\Q)\twoheadrightarrow H_d(\widetilde{X}_{n+1};\Q)_{\cG_{n+1}}$$ factor through the coinvariants $H^d(\widetilde{X}_{n};\Q)_{\cG_n}$ and $H_d(\widetilde{X}_n;\Q)_{\cG_n}$, respectively. Here we denote $\cG_n = G \wr S_n$.)

The relations $\varphi^n \pi_{\#} = \pi_{\#} \tilde{\varphi}^n$, $\varphi_n \pi_{\#} = \pi_{\#} \tilde{\varphi}_n$  and $\pi_{\#}\tau=id$, imply
$$ \varphi^n = \varphi^n (\pi_{\#}\tau) = (\varphi^n \pi_{\#})\tau = (\pi_{\#} \tilde\varphi^n )\tau   $$
$$ \varphi_n = \varphi_n (\pi_{\#}\tau) = (\varphi_n \pi_{\#})\tau = (\pi_{\#} \tilde\varphi_n )\tau.$$
Thus, we can identify the induced maps $(\varphi^n)^*$ and $(\varphi_n)_*$ on (co)homology with the composites 
{\footnotesize \[\begin{tikzcd}
	{H^d({X}_n, \mathbb{Q})} & {H^d(\widetilde{X}_n, \mathbb{Q})^{\mathcal{G}_n}} & {H^d(\widetilde{X}_n, \mathbb{Q})} & {H^d(\widetilde{X}_{n+1}, \mathbb{Q})} & {H^d(\widetilde{X}_{n+1}, \mathbb{Q})_{\mathcal{G}_{n+1}}} & {H^d({X}_{n+1}, \mathbb{Q})} \\
	&& {H^d(\widetilde{X}_n, \mathbb{Q})_{\mathcal{G}_n}}
	\arrow["\cong"', from=1-1, to=1-2]
	\arrow["{\pi^*}", bend left, from=1-1, to=1-3]
	\arrow[hook, from=1-2, to=1-3]
	\arrow["\cong"', from=1-2, to=2-3]
	\arrow["{(\tilde\varphi^n)^*}", from=1-3, to=1-4]
	\arrow[two heads, from=1-3, to=2-3]
	\arrow[two heads, from=1-4, to=1-5]
	\arrow["{\tau^*}", bend left, from=1-4, to=1-6]
	\arrow["\cong"', from=1-5, to=1-6]
	\arrow["{(\tilde\varphi^{n})^*_{\mathcal{G}}}"', dashed, from=2-3, to=1-5]
\end{tikzcd}\] } 

{\footnotesize  \[\begin{tikzcd}
	&&& \; \\
	{H_d({X}_n, \mathbb{Q})} & {H_d(\widetilde{X}_n, \mathbb{Q})^{\mathcal{G}_n}} & {H_d(\widetilde{X}_n, \mathbb{Q})} & {H_d(\widetilde{X}_{n+1}, \mathbb{Q})} & {H_d(\widetilde{X}_{n+1}, \mathbb{Q})_{\mathcal{G}_{n+1}}} & {H_d({X}_{n+1}, \mathbb{Q})} \\
	&& {H_d(\widetilde{X}_n, \mathbb{Q})_{\mathcal{G}_n}}
	\arrow["\cong"', from=2-1, to=2-2]
	\arrow["{\tau_*}", bend left, from=2-1, to=2-3]
	\arrow[hook, from=2-2, to=2-3]
	\arrow["\cong"', from=2-2, to=3-3]
	\arrow["{(\tilde\varphi_n)_*}", from=2-3, to=2-4]
	\arrow[two heads, from=2-3, to=3-3]
	\arrow[two heads, from=2-4, to=2-5]
	\arrow["{\pi_*}", bend left, from=2-4, to=2-6]
	\arrow["\cong"', from=2-5, to=2-6]
	\arrow["{((\tilde\varphi_n)_{*})_{\mathcal{G}}}"', dashed, from=3-3, to=2-5]
\end{tikzcd}\] } 

In particular, we can express the map $(\varphi^n)^*$ as a composite of the induced map $(\tilde\varphi^n)^*_{\cG}$ on coinvariants with several other maps that are all isomorphisms. The analogous statement holds for $(\varphi_n)_*$. 
We conclude the following result. 

\begin{prop} \label{InducedMapCoinvariants} With notation as in the previous paragraphs, the map 
$$(\varphi^n)^*: H^d(X_n;\Q) \longrightarrow H^d(X_{n+1};\Q)$$ is an isomorphism (or injective, surjective, respectively) on cohomology precisely when the induced map $$(\tilde\varphi^n)^*_{\cG}: H^d(\widetilde{X}_n;\Q)_{\cG_n} \longrightarrow  H^d(\widetilde{X}_{n+1};\Q)_{\cG_{n+1}}$$ is an isomorphism (or injective, surjective, respectively) on cohomology coinvariants. The map 
$$(\varphi_n)_*: H_d(X_n;\Q) \longrightarrow  H_d(X_{n+1};\Q)$$ is an isomorphism (or injective, surjective, respectively) on homology precisely when the induced map $$((\tilde\varphi_n)_*)_{\cG}: H_d(\widetilde{X}_n;\Q)_{\cG_n} \longrightarrow  H_d(\widetilde{X}_{n+1};\Q)_{\cG_{n+1}},$$ is an isomorphism (or injective, surjective, respectively) on homology coinvariants. 
\end{prop}

In our setting, we will see that representation stability for the sequences $$\left((\tilde\varphi^n)^*: H^d(\widetilde{X}_n;\Q)\to   H^d(\widetilde{X}_{n+1};\Q)\right)_n \text{ and } \left((\tilde\varphi_n)_*: H_d(\widetilde{X}_n;\Q)\to  H_d(\widetilde{X}_{n+1};\Q)\right)$$ will imply that the induced maps on coinvariants $(\tilde\varphi^n)^*_{\cG}$, $((\tilde\varphi_n)_*)_{\cG}$ are isomorphisms in a stable range, and hence by Proposition \ref{InducedMapCoinvariants} deduce (co)homological stability for the sequences $$\Big((\varphi^n)^*: H^d(X_n;\Q) \to H^d(X_{n+1};\Q)\Big)_n \text{ and }\Big((\varphi_n)_*: H_d(X_n;\Q) \to  H_d(X_{n+1};\Q)\Big)_n.$$ 

The following result is a special case of, for example, Gadish \cite[Theorem 4.5]{Gadish-FItype}. In the case of $\fig$-modules it is also possible to deduce the statement directly, using the description of the induced modules (Proposition \ref{M(m)Facts} part (\ref{M(T)Induced})), and the formula for their decomposition into irreducible constituents (stated in Corollary \ref{PieriRuleWreath}).

\begin{prop}[{Gadish \cite[Lemma 4.3 and Theorem 4.5]{Gadish-FItype}}] \label{NirCoinvariant} Let $G$ be a finite group. Let $V$ be an induced $\fig$-module (in the sense of Definition \ref{DefnInduced}) generated in degree $\leq r$. Then the $\fig$-module structure on $V$ induces a $\fig$-module structure on its sequence of coinvariants $(V_n)_{G \wr S_n}$. All morphisms act on the coinvariants by injective maps, and for $n \geq r$ the induced maps  $$(V_n)_{G \wr S_n} \longrightarrow (V_{n+1})_{G \wr S_{n+1}}$$ are isomorphisms. 
\end{prop}

\subsubsection{Rational (co)homological stability for unordered vertical configuration spaces} 

Consider $\widetilde{\cV}_{n}^k (\R^{p,q})$ as a $((k!)^n \cdot n!)$-sheeted cover of $\cV_{n}^k (\R^{p,q})$ defined by the covering space action of the wreath product $S_k \wr S_n$.
We observe that these covering spaces and our stabilization maps satisfy the framework of Section \ref{SectionRepStabilityImpliesClassicalStability}.
We obtain the following corollary, our classical (co)homological stability result for the rational (co)homology of the unordered vertical configuration spaces.

\begin{corollary} \label{CorClassicalStability}
     Fix  $k \geq 1$, $p \geq 0$, $q \geq 2$, and (co)homological degree $d$.  The (co)homology of the unordered vertical configuration spaces, $\left( H_d({\cV}^k_n (\R^{p,q});\Q) \right)_{n}$ and $\left( H^d({\cV}^k_n (\R^{p,q});\Q) \right)_{n}$, stabilize. Specifically, the maps on homology (respectively, cohomology) induced by the stablization maps $\phi_n: {\cV}^k_n (\R^{p,q}) \to  {\cV}^k_{n+1} (\R^{p,q})$ (respectively, the forgetful maps $\phi^n: {\cV}^k_{n+1} (\R^{p,q}) \to   {\cV}^k_{n} (\R^{p,q})$) are always injective, and are isomorphisms for $n \geq \left\lfloor \frac{2d}{q-1} \right\rfloor$. 
\end{corollary}

\begin{proof}
This result follows by combining our main result Theorem \ref{mainThm} with Propositions \ref{InducedMapCoinvariants} and \ref{NirCoinvariant}.
\end{proof}

\bibliographystyle{amsalpha}
\bibliography{Bibliography}
\quad \\ 

\end{document}